\documentclass[reqno]{amsart}

    \usepackage{hyperref}
    \usepackage{float}
    \usepackage{amsmath} 
    \usepackage{graphicx}
    \usepackage{latexsym} 
    \usepackage{amsfonts}  
    \usepackage[shortlabels]{enumitem}
    \usepackage{amssymb} 
    \usepackage{verbatim}
    \usepackage{color}
    \theoremstyle{plain}
    \newtheorem{theorem}{Theorem}
    \newtheorem{corollary}[theorem]{Corollary}
    \newtheorem{lemma}[theorem]{Lemma}
    \newtheorem{proposition}[theorem]{Proposition}
    \theoremstyle{definition}
    \newtheorem{assumption}{Assumption}

    \newtheorem{remark}[theorem]{Remark}
    \newtheorem*{remark*}{Remark}
    \newenvironment{assump}[2][]
    {\renewcommand\theassumption{#2}\begin{assumption}[#1]}
        {\end{assumption}}

    \newcommand{\rd}{\mathbb R^d}
    \newcommand{\R}{\mathbb R}
    
    \newcommand{\pr}{\mathbf P}
    \newcommand{\e}{\mathbf E}
    \newcommand{\ind}[1]{{\rm \bf1}\{#1\}}

\begin{document}
    \title[ Markov chains in cones]
    {Markov chains in the domain of attraction of 
    Brownian motion in cones} 
    \thanks{
        D. Denisov was supported by a Leverhulme Trust Research Project Grant  RPG-2021-105. 
    }
    \author[Denisov]{Denis Denisov}
    \address{Department of Mathematics, University of Manchester, UK}
    \email{denis.denisov@manchester.ac.uk}
    
    \author[Zhang]{Kaiyuan Zhang}
    \address{Department of Mathematics, University of Manchester, UK}
    \email{kaiyuan.zhang@manchester.ac.uk}
    
    \begin{abstract}
We consider a multidimensional Markov Chain $X$ converging 
to a multidimensional Brownian Motion.  
We  construct a positive  harmonic function for $X$ killed 
on exiting the cone.  
We show that its asymptotic behavior 
is similar to that of to the harmonic function of Brownian motion. 
We use the harmonic function to 
study the asymptotic behaviour of the  tail distribution 
of the exit time $\tau$  of $X$ from a cone.  

\end{abstract}
    
    
    \maketitle

    
    \section{Introduction, assumptions and main results}
    \subsection{Introduction}
    
    Consider a Markov chain $X=(X(n))_{n\ge 0}$ 
    on $\rd$, $d\geq1$, where
     $X(n)= (X_1(n),X_2(n),\cdots, X_d(n))$ 
     is a random vector. 
    Denote by $\mathbb{S}^{d-1}$ the
    unit sphere in $\rd$ centred at the origin and by $\Sigma$ an open and connected subset of
    $\mathbb{S}^{d-1}$. Let $K$ be the cone generated by the rays
    emanating from the origin and passing through $\Sigma$, i.e.
    $\Sigma=K\cap \mathbb{S}^{d-1}$.
    
    Let $\tau$ be the exit time from $K$ 
    of the Markov chain $X$, that is 
    $$
    \tau:=\inf\{n\ge 1\colon X(n)\notin K\}.
    $$

In  a sequence of papers \cite{DW10, DW15, DW19, randomwalkinconesrevisited} 
Denisov and Wachtel  studied 
the asymptotic behavior of  random walks in 
cones.   
In these papers they presented several 
approaches to construction of 
the positive  harmonic function 
for random walks  killed on leaving  the cone $K$.  
Then the harmonic function was  used to 
find asymptotics for 
    \[
    \pr_x(\tau>n),\quad n\to \infty,
    \]
to construct random walks conditioned to stay in the cone, 
to prove conditional global and local limit theorems 
and establish convergence of  random walk conditioned to 
stay in the cone  to Brownian motion conditioned to stay in the cone.  
Similar arguments work in a number of  different situations. 
Some Markovian examples can be found in e.g.~\cite{DKW, DW15b, GLL16,GLP16}.    
The methodology works well for time-dependent boundary as well, see~\cite{DSW16} 
and follow-up articles.  

The key role in the above construction 
was played by the KMT coupling (also called strong coupling, Hungarian construction, etc.) 
of random walks with Brownian motion. 
While it is  possible to construct a harmonic function 
without this coupling, the analysis of the tail asymptotics 
and subsequent limit theorems made an essential use of 
this strong approximation.  
The use of the KMT coupling 
presents serious difficulties in the analysis of other Markovian situations, 
as the KMT coupling heavily relies 
on independence of increments of random walks.  
On the other hand a  Functional Central Limit Theorem 
is available in a number of situations, as it holds for example, for 
martingales. 
In general, it is also easier to prove a FCLT then 
to find a strong coupling with  a specific rate of convergence.

    In the present paper we  develop further  the methodology 
    of~\cite{randomwalkinconesrevisited}   
    to consider the Markov chains in the domain of attraction of  
    Brownian motion.   
    In this case construction of harmonic function remains essentially 
    the same as in~\cite{randomwalkinconesrevisited}.  
    The main improvement of the present paper 
    is in finding the  asymptotics 
    for the tail of $\tau$ and  global conditional limit theorem.  
    Here we do not use a strong coupling and instead 
    rely  on FCLT and large deviations estimates for martingales.  
    In fact it is not even clear if a version of a strong coupling with 
    sufficiently good rate of convergence exists under our conditions.  
    In this paper we have not considered an interesting question how to prove 
    FCLT for conditioned processes without strong coupling, 
    but plan to study this question in future. 
    
    There are a number of cases when   the methodology 
    based on FCLT instead of the KMT coupling  
    might be of interest.  
    For example, a natural   question 
    is how to study  
    the long-time  behaviour of 
     random walks which belong to  the domain of attraction of stable 
     processes in cones. 
     In this case KMT coupling is not available because of big jumps 
     while FCLT is well-known. 
      Currently the first author investigates 
      this situation in the joint project~\cite{CDWP23}.  
    Another interesting question is to describe 
    the behaviour of Markov chains converging to more general diffusion 
    processes rather than standard Brownian motion. 
    For Lamperti Marjov chains in one dimensions this question has been considered 
    in~\cite{DKW}. 
    In future work we are also planning to consider 
     Markov chains in the domain of attraction of 
      more general multidimensional diffusion processes.  
    
    It is worth mentioning that 
     asymptotic behaviour of killed random walks and Markov chains was   
      considered by Varoupolous in a series of works~\cite{Var99,Var00,Var01}. 
      In these works estimates for the 
      heat kernel $\pr_x(X(n)\in dy,\tau>n)$   
        have been obtained and it was shown that  
        the heat kernel of 
        random walks is well approximated  
        with that of Brownian motion.  
        Recently for Markov chains with uniformly bounded increments satisfying uniform ellipticity condition 
    a harmonic function has been constructed in~\cite{MustaphaSifi19} in general Lipschitz domains.  
    In comparison with these 
    work we do not  
    assume  boundedness of increments of Markov chains.  
    The unboundedness of increments presents serious   difficulties 
    in finding asymptotics for $\pr_x(\tau>n)$, 
    as one should carefully analyse the  large deviations.

\subsection{Assumptions and main results}
    
For Markov chains and random walks converging to Brownian motion an important role in our approach 
is played by the positive harmonic function $u$ of the Brownian motion killed at the
    boundary of $K$.  
    This function  can be described as the unique  (up to a constant factor), 
    strictly positive  on $K$ solution of the
    following boundary problem:
    \begin{equation}
    \label{eqn:harmonic_con}
    \begin{cases}
        \Delta u(x)=0, &  x \in K \\
        u(x)=0,& x \in \partial K, 
    \end{cases}      
\end{equation}
where $\Delta$ is the Laplace operator. 
The resulting harmonic  function is homogeneous of a certain order $p>0$, that is 
    $u(x) = |x|^p u(x/|x|)$ for  $x\in K$. 
    When $d=1$  there are  only two non-trivial cones: 
    $K=(0,\infty)$ and $K=(-\infty,0)$.
    For  these cones the positive (on $K$) 
    harmonic function is given by  $u(x)=x$ and  $u(x)=-x$ respectively. 
    It is clear that these functions are homogeneous of  order $p=1$. 

    For $d\ge 2$ 
    the function $u(x)$ and the constant $p$ 
    can be found as a solution to the following eigenvalue problem. 
    Let $L_{\mathbb{S}^{d-1}}$ be the Laplace-Beltrami operator on
    $\mathbb{S}^{d-1}$ and assume that $\Sigma$ is regular with respect to $L_{\mathbb{S}^{d-1}}$.
    Then  there exists a complete set of orthonormal eigenfunctions
    $m_j$  and corresponding eigenvalues $0<\lambda_1<\lambda_2\le\lambda_3\le\ldots$ satisfying
    \begin{align}
     \label{eq.eigen}
      L_{\mathbb{S}^{d-1}}m_j(\sigma)&=-\lambda_jm_j(\sigma),\quad \sigma\in \Sigma\\
      \nonumber m_j(\sigma)&=0, \quad \sigma\in \partial \Sigma. 
    \end{align}
    Then
    \[
    p=\sqrt{\lambda_1+(d/2-1)^2}-(d/2-1)>0
\]
    and the positive harmonic function $u(x)$ 
    solving~\eqref{eqn:harmonic_con}  
    is given by
    \begin{equation}
    \label{u.from.m}
    u(x)=|x|^pm_1\left(\frac{x}{|x|}\right),\quad x\in K.
    \end{equation}
    Note that~\eqref{u.from.m} implies that 
    \begin{equation}\label{eqn:u.main.bound}
        u(x)\le C|x|^p, \quad x\in K, 
    \end{equation}    
    for some constant $C$. 
    We refer to~\cite{BS97,DeB87} for some further details 
    about the harmonic function $u$ and  
    properties of the exit times of Brownian motion
    from a general cone.


Before stating our results we formulate the assumptions that will be 
used in the papers. 
We will denote the distance between $x$ and the boundary  of cone 
$\partial K$ by $d(x)$. As usual,  $|x|$ is the distance from $x$ to the origin. 
We will impose the same geometric assumption 
on the cone as in~\cite{randomwalkinconesrevisited}. 
\begin{assump}{(G)}
Cone $K$ is Lipschitz and star-like.  
There exists a finite constant $C>1$ such that,
\begin{align}
    u(x) & \le C|x|^{p-1} d(x),\quad  x\in K
    \label{eqn:u.harmonic.boundary}\\ 
    u(x)&  \ge C^{-1}|x|^{p-1} d(x),\quad  x\in K. 
    \label{eqn:u.harmonic.boundary.low}
\end{align}
\label{assumption:k}
\end{assump}
\begin{remark} 
We call this assumption geometric since 
a certain regularity is required from the boundary 
of the cone for this assumption to hold.     
It was explained in~\cite{randomwalkinconesrevisited} that  Assumption~\ref{assumption:k} holds when 
     $\Sigma$ is $C^{1,\alpha}$ for 
    $\alpha\in(0,1]$, see~\cite{GT2001} for the definition of $C^{1,\alpha}$. 
    In particular $C^2$ smoothness of $\Sigma$ is sufficient 
    for (G) to hold.  However, this  assumption is slightly 
    stronger than $C^1$ smoothness. 
    \end{remark}    
We will  need some sort of uniform integrability 
from the increments of $X$ to control big jumps. 
For that  we will  assume that 
the increments of Markov chain $X$ are stochastically 
majorised by a random variable satisfying certain moment conditions. 
\begin{assump}{(SM)}
There exists a positive random variable $Y$ 
stochastically majorising the increments of $X$, 
that is 
\begin{equation}
  \sup_{x\in K} \pr_x(|X(1)-x|>t) \le \pr(Y>t),
  \quad \text{ for all } t\ge 0,   
\end{equation}
where the random variable  $Y$ 
satisfies the following moment conditions,
\begin{equation}\label{eq:Y}
    \begin{cases}
       & \e [ Y^p] < \infty,\text{ for } \quad  p>2  \\
       & \e [  Y^2 \log(1+Y)] < \infty,\text{ for } \quad p =2. 
    \end{cases}      
\end{equation}
\label{assumption:x}
\end{assump}
\begin{remark} 
In Denisov and Wachtel~\cite{randomwalkinconesrevisited} 
the same assumption was made about the increments  
of the random walk. 
They also  showed by examples  
that  in the case $p=2$ if 
the condition $\e[Y^2\log(1+Y)]<\infty$ does not hold 
then asymptotic behavior of $\pr_x(\tau>n)$ as $n\to \infty$
is different from that of Brownian motion. 
\end{remark} 
\begin{remark}\label{rem:slow}
    Using~\cite{D06} it was shown in~\cite[Lemma 10]{randomwalkinconesrevisited} 
    that~\eqref{eq:Y}  
    implies that 
    there exists  
    a slowly varying, monotone decreasing differentiable function  $\gamma(t)$
    such that 
    \[
      \e[Y^p;Y>t] = o(\gamma(t)t^{p-2}),\quad t\to\infty  
    \]
    and 
    \[
      \int_1^\infty \frac{\gamma(x)}{x} dx<\infty.   
    \]
 We will use this function in the rest of the paper.    
\end{remark}    

Next we will require that the mean and covariances of increments  are sufficiently close to those of standard Brownian motion.  
\begin{assump}{(M1)}
Let $\gamma$ be the slowly varying, monotone 
differentiable  function 
from Remark~\ref{rem:slow}. We assume that      
\begin{equation*}
    |\mathbf{E}_x[X(1)-x]|= 
    o\left(\frac{\gamma (d(x))}{d(x)}\right)  
\end{equation*}
and 
\begin{equation}
    \begin{cases}
       &  |\mathbf{E}_x[(X_i(1)-x_i)^2]-1| = o(\gamma(d(x)),  \quad  \text{if } 1\le i=j\le d  \\
       & |\mathbf{E}_x[(X_i(1)-x_i) (X_j(1)-x_j)]| = o(\gamma(d(x)),\quad  \text{ if }  1\le i\neq j\le d 
    \end{cases}     
    \label{ass-incre1}
\end{equation}
as $d(x)\to \infty.$
\label{assumption:increment}
\end{assump}
\begin{remark}
    It follows from Remark~\ref{rem:slow} 
    that we can equivalently formulate this condition 
    by using  $t^{2-p}\e[Y^p;Y>t]$ instead of $\gamma(t)$. 
\end{remark}    
For the tail asymptotics we will require that the mean 
and covariances of the Markov chain coincide  
with those of Brownian motion. 
\begin{assump}{(M2)}
For $x\in K$, 
\begin{equation}
    \begin{cases}
       &  \mathbf{E}_x[(X_i(1)-x_i)^2] = 1,\quad  \text{ when  $i=j$} \\
       & \mathbf{E}_x[(X_i(1)-x_i) (X_j(1)-x_j)] = 0 ,\quad\text{ when  $i\neq j$}
    \end{cases} 
\end{equation}
 \label{ass-incre2}
\end{assump}
Note that, Assumption \ref{assumption:increment} is sufficient  
 to show the existence of harmonic function. 
However, to find the tail asymptotics of $\tau$ we make 
use of  the FCLT  
and large deviations results for martingales. 
For that reason to simplify computations 
we have imposed  
a stronger Assumption~\ref{ass-incre2}, which 
ensures that both $X$  and  $(|X(n)|^2-dn)_{n\ge 0}$ 
are martingales. 
This assumption can certainly be relaxed 
and possibly even a weaker Assumption~\ref{assumption:increment} is sufficient. 
However, the relaxation of this assumption requires   
   much more complicated computations.    
   We plan  to return to the question 
   of  
relaxing~\ref{ass-incre2}  
in future, when we will consider  Markov chains 
in the domain of attraction of  more general diffusion processes.     

We are now in position to present our first result  about the existence 
of harmonic function. 
This result is an extension of 
 in~\cite[Theorem 2]{randomwalkinconesrevisited} 
 to the case of Markov chains. 
\begin{theorem}
Let the Markov chain $X$    
satisfy Assumptions \ref{assumption:k}, \ref{assumption:x} and \ref{assumption:increment}. 
 Then the function 
\begin{equation}\label{eq:harmonic}
    V(x) := \lim_{n \to \infty} \mathbf{E}_x [u(X(n); \tau>n] 
\end{equation}
is finite, positive and harmonic 
for $X$ killed at leaving the cone, that is 
\begin{equation}
    V(x) =\mathbf{E}_x [V(X(1)); \tau>1],\quad    
    x\in K. 
\end{equation}
Furthermore, if $p\ge1$ then 
    $$
    V(x)\sim u(x), \quad\text{for }x\in K, \text{ as } d(x)\to\infty,
    $$
    and 
    \[
        \sup_{x\in K: |x|=o(n^{p/(2(p-1))})}
        \left| 
        \frac{\e_x \left[u(X(n));\tau>n\right] - V(x)}{1+u(x)}
        \right| 
        \to 0,\quad  n\to \infty. 
    \]
    If $p<1$ then 
    \[
    V(x)\sim u(x)\quad\text{for }x\in K\ \text{when } d(x)|x|^{p-1}\to\infty.  
    \] 
\label{thm:harmonic.function}
\end{theorem}
Next we proceed to the tail  asymptotics for $\tau$.
Our next result   extends  
the corresponding result 
in~\cite[Theorem 3]{randomwalkinconesrevisited}.  
\begin{theorem}\label{thm:tau}
Let $X$  satisfies   Assumption~\ref{assumption:k}, \ref{assumption:x} and \ref{ass-incre2}.
Assume that  $p \ge 1$.  
Then the following is true. 
\begin{enumerate}[(i)]
\item  There exist  constants $C$ and $R$ such that,
\begin{equation}
    \mathbf{P}_x(\tau>n) \le C \frac{u(x+Rx_0)}{n^{\frac{p}{2}}}, 
    \quad x\in K. 
\end{equation}
\item Uniformly in $x$ such that $|x|=o(\sqrt n)$, 
\begin{equation}
    \mathbf{P}_x(\tau >n) \sim \varkappa \frac{V(x)}{n^{\frac{p}{2}}}, 
    \quad n\to \infty. 
\end{equation}
\item 
For any compact set $D \subset K$, with $|x|=o(\sqrt n)$,  
\begin{equation}
    \mathbf{P}_x \left( \frac{X(n)}{\sqrt{n}} \in D \mid \tau >n \right) \to  c\int_D u(z)\exp\left\{-\frac{|z|^2}{2}\right\} d z, 
    \quad n\to \infty. 
\end{equation}
\end{enumerate}
\label{thm:p.tau}
\end{theorem}

The proof of 
Theorem~\ref{thm:harmonic.function} is similar 
to~\cite{randomwalkinconesrevisited}. 
We will  provide a part of the proof in more details, 
but will refer to  the reader to~\cite{randomwalkinconesrevisited} 
 in some parts. 
 The main improvement of the present paper is 
 in the proof of Theorem~\ref{thm:tau}. 
 There is no analogue of the KMT coupling 
 for Markov chains under consideration and 
we present a new  approach relying on  FCLT.  
 This is a more robust approach as FCLT is available 
 in many situations, while KMT coupling is very specific for random walks.  
 We expect that this improvement will be of great use in many other situations. 

 The paper is organised as follows. 
In Section~\ref{sec:prelim} we first give 
estimates for harmonic and Green function of Brownian motion. 
Then we estimate the error of the diffusion approximation 
given in~\eqref{eqn:error}. 
In Section~\ref{sec:harmonic} we build on these estimates 
and construct the harmonic function for the Markov  
chain and study its asymptotic properties. 
in Section~\ref{sec:upper} we obtain upper bounds 
for $\pr_x(\tau>n)$. 
In the remaining two sections we present the proofs of Theorems.

\section{Preliminary estimates}\label{sec:prelim}   
\subsection{Estimates for the harmonic function and Green function}
In this subsection we will collect 
the bounds for  harmonic function $u(x)$ and Green function of Brownian motion 
from \cite{randomwalkinconesrevisited} 
for the  convenience of the reader. 

We will start with a simple bound for the distance $d(x)$ of $x\in K$ to the boundary of the cone $K$, 
\begin{equation}
\label{eqn:bound.d}
    d(x) \le |x|, 
\end{equation}
To state the next bound recall the multi-index notation for the partial derivatives 
\[
\frac{\partial^\alpha}{\partial x_\alpha } 
=\frac{\partial^{\alpha_1}}{\partial x_{\alpha_1}}
\cdots  
\frac{\partial^{\alpha_d}}{\partial x_{\alpha_d}},  
\]
where  $\alpha =(\alpha_1, \alpha_2, \cdots,\alpha_d)$ 
and $\alpha_j$ is a non-negative integer for each 
$j=1,\ldots,d$.  
\begin{lemma}
There exists a constant $C=C(d)$ such that for $x \in K$ and $\alpha$: $|\alpha|:= \alpha_1+\cdots+\alpha_d \le 3$,
\begin{equation}
    \left|\frac{\partial^\alpha u(x)}{\partial x_\alpha}\right| \le C \frac{u(x)}{{d(x)}^{|\alpha|}}
    \label{eqn:bound.u}
\end{equation}
\end{lemma}
The following inequalities for 
increments of  $u(x)$ have been proved in~\cite[Lemma~2.3]{DW19}.
\begin{lemma}\label{lem:u-diff}
Assume that equation~\eqref{eqn:u.harmonic.boundary} holds. Let $x, y\in K$. Then, 
there exists a constant $C>0$ such that 
  \begin{equation}
    \label{diff-bound}
    |u(x+y)-u(x)|\le  C|y|\left(|x|^{p-1}+|y|^{p-1}\right)
  \end{equation}
and, for $|y|\le |x|/2$,
\begin{equation}
    \label{diff-bound1}
    |u(x+y)-u(x)|\le  C|y||x|^{p-1}.
  \end{equation}
In addition, when $p<1$,  
\begin{equation}
  \label{diff-bound2}
  |u(x+y)-u(x)|\le  C|y|^{p}. 
\end{equation}
 \end{lemma}
Let  $G(x,y)$ be  the Green function of Brownian motion killed 
on leaving the cone $K$.

\begin{lemma}\label{lem:green.bound0}
    Let the assumption~\ref{assumption:k} holds. 
    Then, for any $A>0$ there exists   a constant $C_A$ such that 
\begin{equation}
 \label{eq:green.function.bound0}
 G(x,y)  \le C_A \widehat G(x,y), 
 \end{equation} 
where 
\begin{equation}\label{eq:ghat}
 \widehat G(x,y)  := 
\begin{cases}
    \frac{u(x)u(y)}{|y|^{d-2+2p}},& |x|\le |y|, |x-y|\ge A|y|\\
     \frac{u(x)u(y)}{|x|^{d-2+2p}},& |y|\le |x|, |x-y|\ge A|y|\\ 
     \frac{u(x)u(y)}{|x|^{p-1}|y|^{p-1}} \frac{1}{|x-y|^d}
                                   &d(y)/2\le |x-y| \le A|y|\\
                                   \frac{1}{|x-y|^{d-2}}I(d>2)
                                   +\log\left(\frac{d(y)}{|x-y|}\right)I(d=2)
                                   & |x-y|<d(y)/2.
 \end{cases}
 \end{equation}
 \end{lemma}
\begin{proof} 
    Only a sketch of the proof for $d=3$ was given 
    in~\cite{randomwalkinconesrevisited} 
    for the first 2 regions of~\eqref{eq:ghat}. 
    For that reason we will elaborate estimates in these regions a little bit more.
    The estimates were derived from~\cite[Remark 3.1]{H06} that states  
    \[
      G(x,y) \le C\frac{u(x)u(y)}{\max(|x|,|y|)^{2p-2}|x-y|^d}.   
    \]
    When $|x|\le |y|, |x-y|\ge A|y|$ we immediately obtain from this inequality , 
    \[
        G(x,y) \le \frac{C}{A^d}
        \frac{u(x)u(y)}{|y|^{2p-2}|y|^d} 
        = \frac{C}{A^d}
        \frac{u(x)u(y)}{|y|^{d-2+2p}}.
    \]
    When $|y|\le |x|, |x-y|\ge A|y|$ 
    we further consider two cases. 
    First, when $|y|\le |x| \le 2|y|, |x-y|\ge A|y|$ we write 
    \[
        G(x,y) \le \frac{C}{A^d}
        \frac{u(x)u(y)}{|x|^{2p-2}|y|^d} 
        = \frac{C2^d}{A^d}
        \frac{u(x)u(y)}{|x|^{d-2+2p}}.
    \]
    Second, when $|x|>2|y|$ we can observe that 
    $|x-y|\ge |x|-|y|> |x|-\frac{1}{2}|x|$. Then, 
    \[
        G(x,y) \le 
        \frac{C2^d}{A^d}
        \frac{u(x)u(y)}{|x|^{d-2+2p}}.
    \]
    The estimates for other regions and $d=2$ 
    are more direct and 
    were explained in~\cite{randomwalkinconesrevisited}. 

\end{proof}    

\subsection{Estimates for the error term}
Let $f$ be defined as follows 
\begin{equation}
    f(x) := \e_x [u(X(1))]-u(x),\quad x\in K. 
    \label{eqn:error}
\end{equation}
Recall that $u$ is harmonic for Brownian motion, 
and the means and covariances of $X$ 
are close to those of standard Brownian motion. 
Then function $f$ describes the error in the diffusion approximation 
of harmonic function of $X$ with harmonic function $u$ of Brownian motion.  
We will use this function to construct the exact positive harmonic function  for Markov chain $X$. 
The first step in this construction is the following estimate.  
\begin{lemma}\label{lemma:f}
Let Assumption \ref{assumption:k}, Assumption \ref{assumption:x}  and Assumption \ref{assumption:increment} hold.  
Then, there exists  a function 
$\widetilde \varepsilon(t)$ 
such that $\widetilde \varepsilon(t) \downarrow 0$  as $t\to \infty$ 
and  
\begin{equation}
    |f(x)| \le\widetilde \varepsilon(d(x))  \beta(x), 
\end{equation}
where 
\begin{equation}\label{defn.beta}
    \beta(x) :=|x|^{p-1} d(x)^{-1} \gamma(d(x))
\end{equation}
and 
$\gamma(t)$ is the slowly varying, monotone decreasing differentiable function 
from Assumption~\ref{assumption:increment}. 
\end{lemma}

\begin{proof}
The proof follows closely the corresponding proof 
for random walk in~\cite{randomwalkinconesrevisited}. 
We give this proof in details so that 
it would be clear how to amend the other 
parts of the construction of superharmonic function in~\cite{randomwalkinconesrevisited}. 

Using  the Taylor expansion, 
\begin{equation}
\label{eqn:taylor}
    \left|u(x+y)-u(x) -\nabla u(x) \cdot y - \frac{1}{2} \sum_{i,j} \frac{\partial^2 u(x)}{\partial x_i \partial x_j} y_i y_j\right| 
    \le R_{2,\alpha} (x) |y|^{2+\alpha}
\end{equation}
where
    \begin{equation}
          R_{2,\alpha}(x)= 
          \sup_{i,j}
          [U_{x_i x_j}]_{\alpha, B(x,d(x)/2)}<\infty. 
    \end{equation}
    Here,  we let for $\alpha\in (0,1]$ and any open set  $D$, 
 \[
    [f]_{\alpha, D} := 
    \sup_{y,z\in D, y\neq z}  \frac{|f(y)-f(z)|}{|y-z|^\alpha}. 
 \]
First split  the expectation~\eqref{eqn:error}  in two parts with  $|X(1)-x|\le d(x)/2$ and $|X(1)-x|> d(x)/2$,
\begin{multline*}
    |f(x)| \le \left| \mathbf{E}_x \left[u(X(1))-u(x); |X(1)-x| \le \frac{d (x)}{2}\right] \right| \\
    + \left| \mathbf{E}_x \left[u(X(1))-u(x); |X(1)-x| > \frac{d (x)}{2}\right]\right|.  
\end{multline*}    
Applying equation~\eqref{eqn:taylor} to the first expectation, 
\begin{align*}
     |f(x)| &\le \left|\mathbf{E}_x \left[ \nabla u(x) \cdot (X(1)-x) ; |X(1)-x| \le \frac{d (x)}{2}\right] \right|\\ 
     &+ \left|\mathbf{E}_x \left[ \frac{1}{2} \sum_{i,j} \frac{\partial^2 u(x)}{\partial x_i \partial x_j} (X_i(1)-x_i)(X_j(1)-x_j) ; |X(1)-x| \le \frac{d (x)}{2}\right]\right| \\ 
     &+ \left| R_{2,\alpha}(x)  \mathbf{E}_x\left [|X(1)-x|^{2+\alpha};|X(1)-x|  \le \frac{d (x)}{2}\right] \right| \\&+ \left| \mathbf{E}_x \left[u(X(1))-u(x) ; |X(1)-x| > \frac{d (x)}{2}\right] \right|. 
\end{align*}  
Then we can rewrite the equation above as,
\begin{align}
    \nonumber |f(x)|
    &\le |\nabla u(x)||\mathbf{E}_x[X(1)-x]|+|\nabla u(x)| \left|\mathbf{E}_x \left[|X(1)-x| ; |X(1)-x| > \frac{d (x)}{2}\right]\right| 
    \\\nonumber &+ \left|\mathbf{E}_x \left[ \frac{1}{2} \sum_{i,j} \frac{\partial^2 u(x)}{\partial x_i \partial x_j} (X_i(1)-x_i)(X_j(1)-x_j)\right]\right|
    \\\nonumber &+ \left|\mathbf{E}_x \left[ \frac{1}{2} \sum_{i,j}\frac{\partial^2 u(x)}{\partial x_i \partial x_j} (X_i(1)-x_i)(X_j(1)-x_j) ; |X(1)-x| > \frac{d (x)}{2}\right]\right|
    \\\nonumber  &+ \left|R_{2,\alpha}(x) \mathbf{E}_x\left[|X(1)-x|^{2+\alpha};|X(1)-x| \le \frac{d(x)}{2}\right] \right| \\&+ \left| \mathbf{E}_x \left[u(X(1))-u(x) ; |X(1)-x| > \frac{d (x)}{2}\right] \right]. 
\label{eqn:taylor.stochastic}
\end{align}
Applying the bound~\eqref{eqn:bound.u} for  partial derivatives  and Assumption~\ref{assumption:increment} 
 to the first two terms we obtain 
for some $\widetilde \varepsilon_0(t)\downarrow 0$, as 
$t\uparrow \infty$, 
\begin{align*}
    &|\nabla u(x)||\mathbf{E}_x[X(1)-x]|+|\nabla u(x)| \left|\mathbf{E}_x \left[|X(1)-x| ; |X(1)-x| > \frac{d (x)}{2}\right]\right| \\ 
    &\hspace{0.5cm}\le
    \widetilde \varepsilon_0(d(x))   |x|^{p-1} \gamma (d(x))d(x)^{-1} + c |x|^{p-1} \left|\mathbf{E}_x \left[|X(1)-x| ; |X(1)-x| > \frac{d (x)}{2}\right]\right| \\ 
    &\hspace{0.5cm} \le
    \widetilde \varepsilon_0(d(x)) \beta(x) +  c |x|^{p-1} \int_{\frac{d(x)}{2}}^{\infty} y d \mathbf{P}_x(|X(1)-x|<y). 
\end{align*}    
Making use of the Karamata theorem and the Assumption~\ref{assumption:x} we obtain  the following bound for the first part of equation~\eqref{eqn:taylor.stochastic},
\begin{align*}
      &|\nabla u(x)||\mathbf{E}_x[X(1)-x]|+|\nabla u(x)| \left|\mathbf{E}_x \left[|X(1)-x| ; |X(1)-x| > \frac{d (x)}{2}\right]\right| \\
      &\hspace{0.5cm}\le \widetilde \varepsilon_0(d(x)) \beta(x) +  c |x|^{p-1} \int_{\frac{d(x)}{2}}^{\infty} \mathbf{P}_x(|X(1)-x|>y) dy \\ 
      &\hspace{0.5cm}\le  \widetilde \varepsilon_0(d(x)) \beta(x) +  c |x|^{p-1} \int_{\frac{d(x)}{2}}^{\infty} y^{-2} \e\left[Y^2; Y>y\right] dy \\ 
      &\hspace{0.5cm}\le  \widetilde \varepsilon_0(d(x)) \beta(x) +  c |x|^{p-1} \int_{\frac{d(x)}{2}}^{\infty} y^{-2} \gamma(y) dy 
      \le \widetilde \varepsilon_1(d(x)) |x|^{p-1} d(x)^{-1} \gamma(d(x))  
\end{align*}
for some $\widetilde \varepsilon_1(t)\downarrow 0$  as 
$t\uparrow\infty$. 

Next we estimate the term with the second order derivative. First we make  use of the harmonicity of $u(x)$ to estimate 
the  first term with the second partial derivatives 
in equation~\eqref{eqn:taylor.stochastic}.
Then it follows from equation~\eqref{ass-incre1} and the Markov inequality,
\begin{align*}
   & \left|\mathbf{E}_x \left[ \frac{1}{2} \sum _{i,j} \frac{\partial ^2 u}{\partial x_i \partial x_j} (X_i(1) -x_i)(X_j(1) -x_j) |; |X(1)-x| \le \frac{d(x)}{2}\right] \right|  \\ 
   & \le
     \left|\frac{1}{2} \mathbf{E}_x \left[ \sum _{i \neq i } u_{{x_i}{x_j}} (X_i(1) -x_i)(X_j(1) -x_j)\right]\right| + \widetilde \varepsilon_2(d(x)) |x|^{p-1}d(x)^{-1} \gamma(d(x)) \\ &+\int ^\infty _ {\frac{d(x)}{2}} y^2 d \mathbf{P}_x ( |X(1) -x| <y)  \le
     \widetilde \varepsilon_3(d(x)) |x|^{p-1}d(x)^{-1} \gamma(d(x))
\end{align*}
for some $\widetilde \varepsilon_2(t),\widetilde \varepsilon_3(t)\downarrow 0$
Then we use equations \eqref{eqn:u.main.bound}, \eqref{eqn:u.harmonic.boundary} and~\eqref{eqn:bound.u} to bound  $R_{2,\alpha}(x)$ as follows 
\[
R_{2,\alpha}(x) \le c_\alpha |x|^{p-1} d(x)^{-\alpha-1},\quad \alpha \in (0,1]. 
\]
Using Remark~\ref{rem:slow} 
\begin{align*}
& \mathbf{E}_x \left[|X(1)-x|^{2+\alpha}, |X(1)-x| < \frac{d(x)}{2}\right] \\ 
&\le \int ^{\frac{d(x)}{2}} _{0} (2+\alpha) y^{1+\alpha} \mathbf{P}_x(|X(1)-x| >y) dy \le
\int ^{\frac{d(x)}{2}} _{0} (2+\alpha) y^{1+\alpha} \mathbf{P}(Y >y) dy \\ 
&\le (2+\alpha) \int ^{\frac{d(x)}{2}} _{0} y^{\alpha-1} \mathbf{E}\left[Y^2, Y>y\right] dy \le
c\int ^{\frac{d(x)}{2}} _{0} 
\widetilde \varepsilon_4(y)
y^{\alpha-1} \gamma(y) d y \\ 
& \le  
\widetilde \varepsilon_5(d(x))
d (x)^{\alpha} \gamma(d(x))
\end{align*}
for some slowly varying $\widetilde \varepsilon_4(t),\widetilde \varepsilon_5(t)\downarrow 0$. 

Using equation~\eqref{diff-bound} to bound the last term in equation~\eqref{eqn:taylor.stochastic},
\begin{align*}
    & \left|\mathbf{E}_x \left[u(X(1))- u(x); |X(1)-x| > \frac{d(x)}{2}\right]\right|\\ 
    &\hspace{1cm} \le 
    c|x|^{p-1} \mathbf{E}_x\left[|X(1)-x|; |X(1)-x| > \frac{d(x)}{2}\right] \\ 
    &\hspace{1.5cm}
    + c \mathbf{E}_x \left[ |X(1)-x|^p; |X(1)-x| > \frac{d(x)}{2}\right]  \le
    \widetilde \varepsilon_6(d(x)) |x|^{p-1} d(x)^{-1} \gamma(d(x)) 
\end{align*}
for some $\widetilde \varepsilon_6(t)\downarrow 0$. 
Therefore, it can be concluded that for all $x \in K$,
\begin{equation*}
    \mathbf{E}_x[u(X(1))- u(x)] \le \widetilde \varepsilon(d(x)) |x|^{p-1} d(x)^{-1} \gamma(d(x)). 
\end{equation*}
for some $\widetilde \varepsilon(t)\downarrow 0$ as $t\uparrow\infty$. 
\end{proof}

\section{Construction of harmonic function}\label{sec:harmonic} 
\subsection{Construction of a superharmonic function} 
Construction of a non-negative superharmonic function is is the same as  for  the random walk in~\cite{randomwalkinconesrevisited}. 
Let
\[
U_{\beta} (y)= \int_K G(x,y) \beta(x) dx,\quad y\in K  \] 
and $U_{\beta} (y)=0,y\notin K$. 
Then $U_{\beta}$ satisfies
\begin{equation}\label{eq:laplacian.drift}
    \Delta_\beta U(y) = -\beta(y), \quad y\in K.
\end{equation} 
Following the same proof as in~\cite{randomwalkinconesrevisited}  with modification similar to Lemma~\ref{lemma:f} we can show that the under Assumption~\ref{assumption:k}, $U_\beta(y)$ is finite, and, moreover, for any $y \in K \colon d(y)>R$,
\begin{equation}
\label{lemma:bound.u_beta}
    U_\beta (y) \le \varepsilon (R) u(y),
\end{equation}
where $\varepsilon(R)$  is monotone decreasing in $R$ to $0$. 
This bound is obtained by  dividing the region of integration into four parts according to estimates  in  Lemma~\ref{lem:green.bound0}.

Similarly to Lemma~11 in~\cite{randomwalkinconesrevisited} 
we can prove following inequality for $U_\beta$.  
\begin{lemma}\label{lem:u.beta-diff}
  Assume that equation~\eqref{eqn:u.harmonic.boundary} holds. 
  Let $x\in K$. Then, 
  \begin{equation}
    \label{diff-bound.ubeta}
    \left|U_\beta(x+y)-U_\beta(x)\right|\le  C|y|\left(|x|^{p-1}+|y|^{p-1}\right)
  \end{equation}
and, for $|y|\le |x|/2$,
\begin{equation}
    \label{diff-bound1.ubeta}
    |U_\beta(x+y)-U_\beta(x)|\le  C|y||x|^{p-1}.
  \end{equation}
For $p<1$ and $x\in K$, 
\begin{equation}
  \label{diff-bound2.ubeta}
  \left|U_\beta(x+y)-U_\beta(x)\right|\le  C|y|^{p}. 
\end{equation}
 \end{lemma}
 This  proof follows closely the proof of  \cite[Lemma~2.3]{DW19}.  
 To establish the Lemma we make use of 
 the estimates for the first derivative of 
 $U_\beta$ from~\cite{randomwalkinconesrevisited}.

Exactly as in~\cite{randomwalkinconesrevisited}  
we define the error term for $U_\beta(y)$ by putting 
\[ 
   f_\beta (x) =\mathbf{E}_x [U_\beta (X(1)) ] - U_\beta (x).  
\] 
\begin{lemma}
For any $\varepsilon >0$, there exists $R>0$ such that for all $x \in K\colon  d(x) >R$, 
\begin{equation}
    \left|f _\beta(x) + \frac{1}{2} \beta(x)\right| \le \varepsilon \beta(x)
\end{equation}
\label{lemma:f_beta}
\end{lemma}
Proofs of equation \eqref{lemma:bound.u_beta} and Lemma \ref{lemma:f_beta},  
are practically the same as in the random walk case~\cite{randomwalkinconesrevisited} and hence omitted.  
Required modification have been demonstrated in  Lemma~\ref{lemma:f}.  

Once we have bounds for errors term $f_\beta(x)$ we can proceed 
to the construction of a positive supermartingale. 
The construction is slightly different from~\cite{randomwalkinconesrevisited}. 
Let 
\[
    V_\beta(x) = u(R x_0 +x)+ 3U_\beta (R x_0 +x).   
\]
\begin{lemma}\label{lemma:supermartingale}
There exists $R>0$ such that 
$V_\beta(x)$ is a non-negative superharmonic function 
in $K$, that is $(V_\beta(X(n))\ind{\tau>n})_{n\ge 0}$ is  
a non-negative supermartingale. 
Moreover, 
\begin{equation}\label{sum.beta}
    \mathbf{E}_x \left[\sum^{\tau-1}_{k=0} \beta(Rx_0 +X(k))\right] \le 2 V_\beta(x)  
\end{equation}
and  there exists a function $\varepsilon(R)\to 0$ such that 
\[
    \mathbf{E}_x \left[\sum^{\tau-1}_{k=1} |f (X(k)+Rx_0)|\right] 
    \le     \varepsilon (R) u(x+Rx_0). 
\]
\end{lemma}
\begin{proof}
Note that by Lemma~\ref{lemma:f_beta} with  
$\varepsilon=\frac{1}{4}$ and a 
sufficiently large $R$, 

\[
\e_x[V_\beta(X(1))-V_\beta(x)]
\le 
6f_{\beta}(Rx_0+x)+
\beta(Rx_0+x)  
\le 
-\frac{\beta(Rx_0+x)}{2}.
\]
Then we have,
\begin{align*}
    &\mathbf{E}_x \left[ V_\beta (X(n)) \ind{\tau > n}-V_\beta (X(n-1)) \ind{\tau > n-1} \mid \mathcal{F}_{n-1}\right] 
    \\
     & \le \mathbf{E}_x[ V_\beta (X(n)) -V_\beta (X(n-1))  \mid \mathcal{F}_{n-1}] \ind{\tau > n-1} \\ 
     &\le -\frac 12 \beta(Rx_0+X(n-1)) \ind{\tau > n-1} ,
\end{align*}
proving the supermartingale property. 
Note also 
\begin{align*}
    0\le \mathbf{E}_x \left[ V_\beta (X(n)) \ind{\tau > n}\right] 
    \le V_\beta(x) - \frac 12 \sum_{k=0}^{n-1} 
    \beta(Rx_0+X(k)) \ind{\tau > k}. 
\end{align*}    
Hence, 
\[
    \sum_{k=0}^{n-1} 
    \beta(Rx_0+X(k)) \ind{\tau > k} 
    \le 2 V_\beta(x). 
\]
Letting $n\to\infty$ we obtain~\eqref{sum.beta}.  

Finally, recall that 
 $0 \le |f(x)| \le 
 \widetilde \varepsilon(d(x))\beta(x)$ 
  by Lemma~\ref{lemma:f}.  
Then, by  equation~\eqref{lemma:bound.u_beta}, 
\begin{align*}
\mathbf{E}_x \left[\sum^{\tau-1}_{k=1} |f (X(k)+Rx_0)|\right] 
    &\le 
    \widetilde \varepsilon(d(CR))
    \mathbf{E}_x \left[\sum^{\tau-1}_{k=1} \beta (X(k)+Rx_0)\right] \\
    & \le \widetilde  \varepsilon(d(CR)) (2u(Rx_0+x)+6U_\beta (x+Rx_0))  
    \\& \le \varepsilon (R) u(x+Rx_0). 
\end{align*}
\end{proof}
\subsection{Construction of the harmonic function}
We are ready to construction the harmonic function. 
\begin{lemma}\label{lem:harmonic} 
Function $V$ in~\eqref{eq:harmonic} is well-defined and is equal to 
\[
V(x) = 
u(Rx_0+x)-\mathbf{E}_x[u(Rx_0+X(\tau)]+\mathbf{E}_x\left[\sum^{\tau-1}_{k=0} f(Rx_0+X(k)\right]. 
\]
This function is positive and harmonic. 
\end{lemma} 
\begin{proof} 
Define a  process $(Z_n)_{n\ge 0}$ as follows,
\[ 
    Z_n = u(Rx_0+X(n \wedge \tau))-\sum^{n \wedge \tau-1} _{k=0} f(Rx_0+X(k)), \quad n\ge 0. \] 
Process $(Z_n)_{n\ge 0}$ is a martingale since 
\begin{align*}
    & \mathbf{E}_x[Z_n- Z_{n-1} |\mathcal{F}_{n-1}] \\
    &\hspace{1cm} = {\rm 1}\{\tau>n-1\} \mathbf{E}_x [u(Rx_0+X(n))-u(Rx_0+X(n-1))|\mathcal{F}_{n-1}]\\&
    \hspace{1.5cm}-{\rm 1}\{\tau>{n-1}\} f(Rx_0+X(n-1)). 
\end{align*}
Then we can  proceed further with the optional stopping  theorem,
    \begin{align*}
     \mathbf{E}_x[Z_0]&=u(Rx_0+x)=\mathbf{E}_x[Z_n] \\ &=\mathbf{E}_x[u(Rx_0+X(n \wedge \tau)]-\mathbf{E}_x\left[\sum^{n\wedge \tau-1}_{k=0} f(Rx_0+X(k))\right] \\ &= \mathbf{E}_x[u(Rx_0+X(n); \tau>n]+\mathbf{E}_x[u(Rx_0+X(\tau); \tau \le  n] \\&-\mathbf{E}_x\left[\sum^{\tau-1}_{k=0} f(Rx_0+X(k));  \tau\le n \right]-\mathbf{E}_x\left[\sum^{n-1}_{k=0} f(Rx_0+X(k)); \tau>n\right].
    \end{align*}
We can now  rearrange the first term as follows, 
\begin{multline}
       \mathbf{E}_x [u(Rx_0+X(n); \tau>n] = u(Rx_0+x)-\mathbf{E}_x[u(Rx_0+X(\tau); \tau \le  n]\\+\mathbf{E}_x\left[\sum^{\tau-1}_{k=0} f(Rx_0+X(k)); \tau\le n \right] +\mathbf{E}_x\left[\sum^{n-1}_{k=0} f(Rx_0+X(k)); \tau>n\right].  
\label{eqn:martingale}
\end{multline}
Since $u(Rx_0+X(\tau))1\{ \tau\le n\}$ is increasing in  $n$, 
the monotone convergence theorem implies 
\[ 
    \lim_{n \to\infty}\mathbf{E}_x[u(Rx_0+X(\tau)); \tau\le n]=\mathbf{E}_x[u(Rx_0+X(\tau))]. 
\] 
Lemma \ref{lemma:supermartingale} allows us to apply the   dominated convergence theorem to  the expectation $\mathbf{E}_x[\sum^{\tau-1}_{k=0} f(Rx_0+X(k));  \tau\le n ]$ to obtain 
\[ 
    \lim _{n \to  \infty} \mathbf{E}_x\left[\sum^{\tau-1}_{k=0} f(Rx_0+X(k)); \tau\le n \right] = \mathbf{E}_x\left[\sum^{\tau-1}_{k=0} f(Rx_0+X(k))\right].  
\] 
Again by the dominated convergence theorem,  the remaining term in equation \eqref{eqn:martingale}
    \begin{align*}
        & \limsup_{n \to \infty} \left|\mathbf{E}_x\left[\sum^{n-1}_{k=0} f(Rx_0+X(k)); \tau>n\right]\right|\\& \le \limsup_{n \to  \infty} \mathbf{E}_x\left[\sum^{n-1}_{k=0} |f(Rx_0+X(k))|; \tau>n\right]\\& \le \limsup _{n \to \infty} \mathbf{E}_x\left[\sum^{\tau-1}_{k=0} |f(Rx_0+X(k))|; \tau>n\right] =0. 
    \end{align*}
As a result, 
    \begin{align*}
      &  \lim_{n \to  \infty} \mathbf{E}_x[u(Rx_0+X(n); \tau>n] \\& = u(Rx_0+x)-\mathbf{E}_x[u(Rx_0+X(\tau)]+\mathbf{E}_x\left[\sum^{\tau-1}_{k=0} f(Rx_0+X(k)\right]. 
    \end{align*}
Next we will prove that the difference $\mathbf{E}_x[u(Rx_0+X(n)); \tau>n]-\mathbf{E}_x[u(X(n)); \tau>n]$ converges to $0$.
According to the equation \eqref{diff-bound},
\begin{align*}
    & |u(Rx_0+X(n))-u(X(n))|  \le CR (|X(n)|^{p-1}+R^{p-1}) \\& \le CR|X(n)|^{p-1}+CR^p \le CR|Rx_0+X(n)|^{p-1}+CR^{p}. 
\end{align*}
By  Lemma \ref{lemma:supermartingale}, 
\[ 
    \mathbf{E}_x[\beta(Rx_0+X(k); \tau>n] \xrightarrow{n \to  \infty} 0. 
\]
Then, there is an increasing sequence $a_n \uparrow \infty$  such that 
\begin{align*}
    & \mathbf{E}_x[|Rx_0+X(n)|^{p-1}; d(Rx_0+X(n)) \le a_n; \tau>n] \\& \le \frac{a_n}{\gamma(a_n)} \mathbf{E}_x[\beta(Rx_0+X(n); \tau>n] \xrightarrow{n \to \infty} 0
\end{align*}
and 
\begin{align*}
       & \mathbf{E}_x[|Rx_0+X(n)|^{p-1}; d(Rx_0+X(n)) > a_n; \tau>n] \\& \le \frac{c}{a_n} \mathbf{E}_x[u(Rx_0+X(n); \tau>n] \xrightarrow{n \to  \infty} 0. 
\end{align*}
Hence, the first statement  can  be obtained as follows, 
\begin{align}
    \nonumber 
    V(x)  &= \lim_{n \to \infty} \mathbf{E}_x[u(X(n)); \tau>n] =\lim_{n \to \infty} \mathbf{E}_x[u(Rx_0+X(n)); \tau>n] \\&=u(Rx_0+x)-\mathbf{E}_x[u(Rx_0+X(\tau))]+\mathbf{E}_x \left[\sum^{\tau-1}_{k=0} f(Rx_0+X(k)\right].
        \label{eqn:harmonic.function}
\end{align}
Then, by Lemma \ref{lemma:supermartingale},
\[ 
    V(x) \le Cu(Rx_0+x). 
\]

The harmonicity of $V(x)$ 
can now be proved as follows. First, 
\[
    V(x) =\mathbf{E}_x[V(X(n)); \tau>n]. 
\] 
Taking account into the Markov property of $X(n)$, we only need to prove the harmonicity  for $n=1$. Then, 
\[
        \mathbf{E}_x[u(X(n+1)); \tau> n+1] = \int_{K} \pr_x(X(1) \in dy; \tau >1)\mathbf{E}_y[u(X(n)); \tau>n].
\]
Similarly to the last step, using equation \eqref{eqn:harmonic.function} with monotone convergence theorem,
    \begin{align*}
       V(x)&= \lim_{n \to \infty} \mathbf{E}_x[u(X(n+1)); \tau>n+1] \\&=\int_{K} \mathbf{P}_x(X(1) \in dy; \tau>1) \lim_{n \to \infty} \mathbf{E}_y[u(X(n)); \tau>n] \\& =\int_{K}\mathbf{P}_x(X(1) \in dy; \tau>1)V(y)=\mathbf{E}_x[V(X(1)); \tau>1]. 
     \end{align*}
Hence, $V(x)$ is harmonic for $X$.

\end{proof}

\section{Upper bounds  for $\pr(\tau_x>n)$ in the case $p\ge 1$} \label{sec:upper} 


We need to proceed differently from~\cite{randomwalkinconesrevisited} for the proof of an upper bound 
for  $\pr_x(\tau>n)$. 
The reason for that is that we do not have the 
strong coupling at our disposal anymore. 
We precede the main argument with the following result.
\begin{lemma}
\label{lemma:bound.near.boundary}
Let the Assumptions \ref{assumption:k}, \ref{assumption:x} and \ref{ass-incre2} hold. 
Then, for $p \ge 1$, there exist constants $q>0,c>0$ and $A>0$ such that,
\begin{equation}
    \label{eq.bound.near.boundary}
    \sup_{x\in K: d(x)\le \varepsilon \sqrt n}\pr_x(\tau>n) \le c\varepsilon^{q},\quad n\ge \frac{A}{\varepsilon}.  
\end{equation}
\end{lemma}
\begin{remark}
    For convex cones this bound follows 
    immediately from estimates   
    for $\pr_x(\tau>n)$ for half planes, 
    which in turn follow immediately from the 
    one-dimensional case. 
    The main difficulty in the presented proof is 
    to include the  non-convex case. 
    As the  proof is quite lengthy it might be a good idea 
    to skip it on the first reading by considering convex cones only.  
\end{remark}    

\begin{proof}[Proof of  Lemma~\ref{lemma:bound.near.boundary}] 
Fix $\varepsilon>0$ and let $x\in K$ be such that $d(x)\le \varepsilon\sqrt n$. 
We consider separately two cases: $|x|\le  3\sqrt{\varepsilon n}$ and $|x|>    3\sqrt{\varepsilon n}$.

\underline{Case $x: |x|\le 3\sqrt{\varepsilon n}$}. 
Let
\[
\nu_0:=\min\{k\ge 1\colon |X(k)|>
3\sqrt{\varepsilon n}\}. 
\]
By Assumption~\ref{ass-incre2} the 
process $(|X(n)|^2 -dn )_{n\ge 0}$ is a martingale. 
 Then, by Doob's optional stopping theorem, 
\[
\e_x\left[|X(n \wedge \nu_0 )|^2 -d \min(n,\nu_0)\right] 
=|x|^2. 
\]
Hence, 
\begin{multline*}
    d\e_x[\min(n,\nu_0)]
    \le 
    \e_x[|X(n)|^2;\nu_0>n] 
    +\e_x[|X(\nu_0)|^2;\nu_0\le n]\\ 
    \le 9 \varepsilon n +\e_x[|X(\nu_0)|^2;\nu_0\le n]. 
\end{multline*} 
The second expectation can be bounded  
as follows, 
\begin{multline*} 
\e_x[|X(\nu_0)|^2;\nu_0\le n] 
=
\e_x[|X(\nu_0)|^2;\nu_0\le n, |X(\nu_0)|\le 4\sqrt{\varepsilon n}] \\ 
+
\e_x[|X(\nu_0)|^2;\nu_0\le n, |X(\nu_0)|>4\sqrt{\varepsilon n}] 
\\
\le 16\varepsilon n
+\e_x[|X(\nu_0)|^2;\nu_0\le n, |X(\nu_0)|>4\sqrt{\varepsilon n}]. 
\end{multline*}
Using the elementary bound $(a+b)^2\le 2(a^2+b^2)$ we 
bound the latter expectation  as follows, 
\begin{multline*} 
\e_x[|X(\nu_0)|^2;\nu_0\le n, |X(\nu_0)|>4\sqrt{\varepsilon n}] 
\le 
2\e_x[|X(\nu_0-1)|^2;\nu_0\le n]
\\
+2\e_x[|X(\nu_0)-X(\nu_0-1)|^2, 
|X(\nu_0)|>4\sqrt{\varepsilon n}, \nu_0\le n] 
 \\ 
 \le 18\varepsilon n+2\sum_{k=1}^n 
\e_x[|X(k)-X(k-1)|^2, 
|X(k)|>4\sqrt{\varepsilon n}, \nu_0 = k].
\end{multline*}
Now note 
that 
\begin{multline*}
\sum_{k=1}^n 
\e_x[|X(k)-X(k-1)|^2, 
|X(k)|>4\sqrt{\varepsilon n}, \nu_0 = k] \\ 
\le 
\sum_{k=1}^n 
\e_x[|X(k)-X(k-1)|^2, 
|X(k)-X(k-1)|>\sqrt{\varepsilon n}, \nu_0 > k-1] \\ 
\le 
\e[Y^2; |Y|>\sqrt{\varepsilon n}] 
\sum_{k=1}^n  \pr_x(\nu_0>k-1)
\le 
\e[Y^2; |Y|>\sqrt{\varepsilon n}] 
\e[\nu_0\wedge n]. 
\end{multline*} 
Assuming that $\sqrt{\varepsilon n}>A_0$ 
for a  sufficiently large $A_0$ 
we can ensure that 
$\e[Y^2; |Y|>\sqrt{\varepsilon n}]\le \frac14$. 
Then we obtain,  
\[
d\e_x[\min(n,\nu_0)] 
\le \frac{1}{2}\e_x[\min(n,\nu_0)]
+34\varepsilon n
\]
and hence 
\begin{equation}
\label{p_nu0_1}
\pr_x(\nu_0>n)
\le \frac{\e_x[\min(n,\nu_0)] }{n}
\le 
\frac{34}{d-1/2}\varepsilon
\end{equation}
for $\varepsilon n >A_0^2$ and a sufficiently large 
$A_0$.  

Next let $K_0$ be a cone that contains 
$K$ 
in such a way that 
\begin{equation}\label{eq:dist}
\text{dist} (\mathbb S^{d-1}\cap K_0,\mathbb S^{d-1}\cap K)>0. 
\end{equation}
Assume also that $K_0$ satisfies smoothness  assumptions of the present paper. 
We can further define (or redefine) the Markov chain 
on $K\setminus K_0$ that it will satisfy our assumptions. 
Then, there exists a superharmonic function 
$V_0(x)$ such that 
\[
\e_x[V_0(X(1))]\le V_0(x), \quad x\in K_0. 
\]
Note that for a larger cone the harmonic 
function for the killed Brownian motion 
is homogeneous of order $p_0<p$. 
Hence, it follows from~\eqref{eq:dist} and~\eqref{eqn:u.harmonic.boundary.low}  
that 
there exists a constant $c$ such that 
\begin{equation}\label{eq:v0.below}
V_0(x)\ge c_0u_0(R_0x_0+x)\ge \frac{1}{c}|x|^{p_0}
\end{equation} for 
$x\in K\subset K_0$. 
Here $u_0$ is the harmonic function of Brownian motion killed on exiting $K_0$.

Now on the event $\nu_0\le n$ we have 
$|X_{\nu_0}|>3{\sqrt{\varepsilon n}}$. 
Hence, applying the  Markov inequality 
together with~\eqref{eq:v0.below} we obtain,
\begin{align}
    \nonumber 
        \pr_x(\tau>n, \nu_0 \le n) 
       &\le \frac{c \e_x [V_0(X(\nu_0 \wedge n)); \tau>(\nu_0 \wedge n)]}{(3\sqrt{\varepsilon n})^{p_0}} \\ & 
       \le \frac{c V_0(x)}{(3\sqrt{\varepsilon n})^{p_0}} 
       \le c \varepsilon^{p_0/2}. 
       \label{p_nu0_2}
\end{align}
Combination of~\eqref{p_nu0_1} and~\eqref{p_nu0_2} proves the statement in the first case as 
\begin{equation}
\label{eq.first.case}
\pr_x(\tau>n)
\le     
\pr_x(\nu_0>n)
+\pr_x(\nu_0\le n,\tau>n)
\le c(\varepsilon + \varepsilon^{p_0/2}). 
\end{equation}

\underline{Case $x: |x|>3\sqrt{\varepsilon n}$}. We start with some geometric observations. 
Let $x_0 \in \partial K$ 
be a point at the boundary of the cone 
such that $|x-x_0| = d(x)$. 
Denote the open ball of radius $t$ with centre at $x_0$ as 
$B(x_0, t):= \{x \in \R^d\colon |x-x_0| <  t\}$. 

Since $K$ is a Lipschitz cone, we can always find a   right circular cone $\tilde K$   with a sufficiently small opening angle such that for any point $z_0\in K\cap \{z\colon |z|<1\}$ 
a suitable translation and rotation of  
 $\tilde K$ 
will give a cone, which we denote by 
$\tilde K(z_0)$,  with its vertex at $z_0$ such that $\tilde K_h(z_0):=\tilde K(z_0)\cap B(z_0,h)$ lies outside of the cone $K$ with strictly positive distances between any $2$ points of boundaries 
of $\tilde K_h(z_0)$ and $K$ excluding vertex $z_0$.  
There exists a constant $c$ such that 
by taking $h\le 1$ sufficiently small, 
  we can ensure that 
for all $z_0 \in K\cap B(0,1)$  
we have also the following uniform geometric property: 
\begin{equation}\label{eq.geom.1}
\text{dist}(y, \widetilde K(z_0)\setminus K_h(z_0))\ge c h, 
\quad \text{ for }  y\in B(z_0, 3h)\cap K. 
\end{equation}
By scaling of cones we then obtain 
that for $z_0:|z_0|>1$, 
\begin{equation}\label{eq.geom}
\text{dist}(y, \widetilde K(z_0)\setminus K_h(z_0))\ge c h|z_0|, 
\quad \text{ for }  y\in B(z_0, 3h|z_0| )\cap K. 
\end{equation}

Let
\[
\nu_1:=\min\{k\ge 1\colon X(k)\notin 
B(x_0, h\sqrt{\varepsilon n})\cap K\}. 
\]
Using the  optional stopping theorem 
and Assumption~\ref{assumption:x} 
with the  martingale 
$((X(n)-x_0))^2-dn $, we obtain 
\[
d\e_x[\nu_1\wedge n ]
=\e_x[(X(\nu_1)-x_0)^2;\nu_1\le n]
+ \e_x[(X(n)-x_0)^2;\nu_1>n]
-|x-x_0|^2.
\]
For the second expectation we obtain immediately 
\[
    \e_x[(X(n)-x_0)^2;\nu_1>n] 
    \le h^2 \varepsilon n. 
\]
For the first expectation we use the same argument as above 
\begin{multline*}
    \e_x[|X(\nu_1)-x_0|^2;\nu_1\le n]
    \le 
    4 h^2\varepsilon n
    +
    \e_x[|X(\nu_1)-x_0|^2;\nu_0\le n,
    |X(\nu_1)-x_0|>2h\sqrt {\varepsilon n} ]. 
\end{multline*}
Then, 
\begin{multline*} 
    \e_x[|X(\nu_1)-x_0|^2;\nu_1\le n, 
    |X(\nu_1)-x_0|>2h\sqrt{\varepsilon n}] 
    \le 
    2\e_x[|X(\nu_1-1)-x_0|^2;\nu_1\le n]
    \\
    +2\e_x[|X(\nu_1)-X(\nu_1-1)|^2, 
    |X(\nu_1)-X(\nu_1-1)|>h\sqrt{\varepsilon n}, \nu_1\le n] 
    \le 2h^2
    \varepsilon n \\ 
    +2\sum_{k=1}^n 
    \e_x[|X(k)-X(k-1)|^2, 
    |X(k)-X(k-1)|>h\sqrt{\varepsilon n}, \nu_1 = k]\\
    \le 
    2h^2\varepsilon n
    +\e[Y^2;Y>h\sqrt{\varepsilon n}] 
    \e_x[\nu_1\wedge n].
    \end{multline*}
 Then, similarly we obtain for sufficiently large $n$ that 
\[
\e_x[\nu_1\wedge n] \le \frac{7h^2}{d-1/2}\varepsilon n.
\]
Hence,
\begin{equation}
\label{p_nu1_1}
    \pr_x (\tau>n, \nu_1>n) \le \pr_x(\nu_1>n) \le 
    \frac{\e_x[\nu_1\wedge n]}{n} 
    \le c \varepsilon
\end{equation}

Next consider
\begin{multline*} 
\pr_x(\tau>n,\nu_1\le n)
=  \pr_x(\tau>n,\nu_1\le n,
X(\nu_1) \in B(x_0,2h\sqrt{\varepsilon n}))\\ 
+
\pr_x(\tau>n,\nu_1\le n,
X(\nu_1) \notin B(x_0,2h\sqrt{\varepsilon n})).
\end{multline*}
For the second case we again use the same argument, 
\begin{multline}
    \pr_x(\tau>n,\nu_1\le n,
    X(\nu_1) \notin B(x_0,2h\sqrt{\varepsilon n})).
\le     
\sum_{k=1}^n 
    \pr(
    |X(k)-X(k-1)|>h\sqrt{\varepsilon n}, \nu_1 = k) \\    
\le n \pr(Y>h\sqrt{\varepsilon n}) 
\le n \frac{\e[Y^2 \log Y]}{h^2\varepsilon n \log (h\sqrt{\varepsilon n})} 
\le \varepsilon 
\label{p_nu1_2}
\end{multline}
for  $n$ such that $\log( h\sqrt{\varepsilon n})\ge (h\varepsilon)^{-2}$. 

Now let $\widetilde V_{x_0}$ be 
the positive  harmonic function constructed for $X$ in the cone 
$\R^d\setminus \tilde K_{x_0}$. 
Let $p_1$ be the corresponding 
order of homogeneity the harmonic function  for the Brownian motion in the right circular cone $\widetilde K$.  
Now we are going to use the geometric property~\eqref{eq.geom}. 
Note that $X(\nu_1)\in B(x_0,2h\sqrt{\varepsilon n})\subset 
B(x_0,h|x|),
$ where the set inclusion holds since $|x|>3\sqrt{\varepsilon n}$. 
Hence ${\rm dist} (X(\nu_1), \tilde K_{x_0}\setminus B(x_0,h|x_0|))\ge c|x_0|\ge c_1\sqrt{\varepsilon n}$. 
Also since $X(\nu_1) \notin B(x_0,h\sqrt{\varepsilon n}) $ we have 
${\rm dist} (X(\nu_1), \tilde K_{x_0}\cap B(x_0,h|x_0|))\ge c_2\sqrt{\varepsilon n}$.  
Combining those statements we obtain that 
\[
{\rm dist} (X(\nu_1), \tilde K_{x_0}) \ge c_13\sqrt{\varepsilon n}. 
\]

Then, 
\begin{multline} 
    \pr_x(\tau>n,\nu_1\le n,
    X(\nu_1) \in B(x_0,2\sqrt{\varepsilon n}))\\ 
    \le \frac{\e_x[ \widetilde V_{x_0}(X(\nu_1));\tau>n,\nu_1\le n,
    X(\nu_1) \in B(x_0,2\sqrt{\varepsilon n}))}{(\sqrt{\varepsilon n})^{p_1}}\\
    \le \frac{\widetilde V_{x_0}(x)}{{(\sqrt{\varepsilon n})^{p_1}}}
    \le C \frac{(\varepsilon \sqrt{ n})^{p_1}}{(\sqrt{\varepsilon n})^{p_1}} 
    \le C \varepsilon^{p_1/2}.
    \label{p_nu1_3}
\end{multline}

Combining  the equations  \eqref{eq.first.case}, \eqref{p_nu1_1}, 
\eqref{p_nu1_2} and \eqref{p_nu1_3}
we arrive at the conclusion. 
\end{proof} 

\begin{lemma}
\label{lemma:bound.p.tau}
Let the Assumptions \ref{assumption:k}, \ref{assumption:x} and \ref{ass-incre2} hold. Then, for $p \ge 1$, there exists a constant $C$ such that,
\[ 
    \pr_x(\tau>n) \le C \frac{V_\beta(x)}{n^\frac{p}{2}} \le C \frac{u(x+Rx_0)}{n^\frac{p}{2}}. 
\] 
\end{lemma}
\begin{proof}
Fix an $\varepsilon >0$,
\begin{multline*}
    \pr_x (\tau>n) =\pr_x \left(\tau>n, d(X(n/2)) \ge \varepsilon n^{1/2}\right)+\pr_x \left(\tau>n, d(X(n/2)) < \varepsilon n^{1/2}\right)
\end{multline*}
The first term can be bounded by the Markov inequality and Theorem \ref{thm:harmonic.function},
\begin{equation}
\label{eqn:p1}
    \begin{split}
        & \pr_x \left(\tau>n, d(X(n/2)) \ge \varepsilon n^{1/2}\right) \le c \e_x \left[ \frac{u(X(n/2)), \tau>n/2}{ c (\varepsilon n^{1/2})^p}\right] \\ & \le \frac{c}{\varepsilon^{\frac{p}{2}}} \frac{\e_x [V_\beta(X(n/2)), \tau>n/2]}{n^{p/2}} \le \frac{c V_\beta(x)}{n^{p/2}}
    \end{split}
\end{equation}
Then we rewrite the second term,
\begin{equation}
\label{eqn:p2}
\begin{split}
 &\pr_x\left(\tau>n, d(X(n))< \varepsilon n^{1/2}\right) \\& = \int_{d(y) < \varepsilon\sqrt{n}} \pr_x \left(\tau \ge n/2, X(n/2) \in dy\right) \pr_y \left(\tau> n/2\right)
 \end{split}
\end{equation}
Use Lemma \ref{lemma:bound.near.boundary}, equation \eqref{eqn:p2} can be written as,
\[
    \pr_x\left(\tau>n, d(X(n))< \varepsilon n^{1/2}\right) \le f(\varepsilon) \pr_x(\tau>n/2)
\]
With $f(\varepsilon):= c \varepsilon ^{q}$,
combine this with equation \eqref{eqn:p1},
\[
\pr_x(\tau>n) = \frac{c}{\varepsilon^{p/2}}\frac{V_\beta(x)}{n^{p/2}}+f(\varepsilon) \pr_x\left(\tau>\frac{n}{2}\right)
\]
Iterating this estimate $N$ times, we obtain,
\[
\pr_x(\tau>n)\le c \sum^{N-1}_{j=0} \frac{f(\varepsilon)^j}{\varepsilon^{p/2}}\frac{V_\beta(x)}{(n/2)^{p/2}} + f(\varepsilon)^N\pr\left(\tau> n/2^N\right)
\]
Now for appropiate $c_1$, sufficiently small $\varepsilon$ and  $N\sim c_1\log(n)$ 
we obtain 
\[
\pr_x \left(\tau>n\right) \le c\frac{V_\beta(x)}{n^{p/2}}+c\frac{1}{n^{p/2}}
\]
\end{proof}

\begin{corollary}
    \label{lemma:bound.p.tau.1}
    Let the Assumptions \ref{assumption:k} and \ref{assumption:x} hold. Then, for $p \ge 1$, 
    there exists a constant $C$ such that for \(x\in K\), 
    \[ 
        \pr_x(\tau>n) \le C \frac{d(Rx_0+x)}{\sqrt n}. 
    \] 
    \end{corollary}
    \begin{proof}
       For $p=1$ the statement follows from Lemma~\ref{lemma:bound.p.tau} 
       and the estimate~\ref{eqn:u.harmonic.boundary}. 
       For $p>1$ we can first consider a large cone 
       with $p=1$. Then the required estimate will follow from the case $p=1$.    
   \end{proof}    

\section{Proof of Theorem~\ref{thm:harmonic.function}}
 The first part of Theorem~\ref{thm:harmonic.function} 
 was proved in Lemma~\ref{lem:harmonic}. 
 With Lemma~\ref{lemma:bound.p.tau}, the proof of asymptotic property of harmonic function $V$ in Theorem~\ref{thm:harmonic.function} can be completed following the corresponding proof for random walk in \cite{randomwalkinconesrevisited}.

\section{Proof of Theorem~\ref{thm:p.tau}.}

We will preface the proof with the following Proposition. 
\begin{proposition}\label{prop.A}
    There exists a function 
    $f(A)\to 0$, as $A\to\infty$  such that  
    for sufficiently large \(n\),
    \[
        \sup_{x\in K:|x|\le \sqrt n} 
        \frac{\e_x[V_\beta(X(n));\tau>n, |X(n)|>A\sqrt n]}{V_\beta(x)}
        \le f(A). 
    \]
\end{proposition}    
\begin{proof}
Note first that since \(|x|\le \sqrt{n} \), 
 \[
        \e_x[V_\beta(X(n));\tau>n, |X(n)|>A \sqrt n]
\le
        	\e_x[V_\beta(X(n));\tau>n, |X(n)-x|>\widetilde A \sqrt n],
\]
where \(\widetilde A=A-1\). 
Define the sequence \(n_k\) by putting \(n_0=n\), 
\(n_k:=\sup\{2^j: 2^j\le n_{0}\}\)
and \(n_k:=n_{k-1}/2\) for \(k\ge 2\).  
We assume that   \(k\in [0, m(n)]\), where \(m(n)\) is 
picked  \(m=m(n)\) in such a way 
that \(n_m= 1\).  
Put $c_0= 1$ and let  
$c_{k+1}=\frac{3}{4} c_k \sqrt{\frac{n_k}{n_{k+1}}}$. 
Note that for \(k\ge 1\) we have 
\(c_{k+1} = c_1 (3\sqrt{2}/4)^{k}\uparrow \infty\).

We will now construct 
 estimates 
 for 
 \[
    E_{n,k}(x):=\e_x\left[V_\beta(X(n_k));\tau>n_k, |X(n_k)-x|>\widetilde A c_k \sqrt{n_k}\right]   
 \]
 recursively and will describe now the recursion step from \(k\) to \(k+1\). 
Consider the following sequence of events, 
\begin{align*}
    C_{n,k}&:= \left\{|X(n_{k+1})-X(0)|>\frac{3}{4}\widetilde A c_k\sqrt {n_k}\right\}
       =\left\{|X(n_{k+1})-X(0)|>\widetilde A c_{k+1}\sqrt {n_{k+1}}\right\}\\ 
    D_{n,k}&:= \left\{|X(n_{k})-X(n_{k+1})|> \frac{1}{4}\widetilde Ac_k\sqrt{n_{k}}\right\}.  
\end{align*}    

We have,  
\[
    E_{n,k}(x)\le E_1+E_2,     
\]
where
    \begin{align*} 
	E_1&:= \e_x\left[V_\beta(X(n_{k}));\tau>n_{k}, C_{n,k} \right]\\
	E_2&:=\e_x\left[V_\beta(X(n_{k}));\tau>n_{k}, C_{n,k}^c,D_{n,k}\right].  
    \end{align*} 
    Using the supermartingale property of $V_\beta(X(n))\ind{\tau>n}$ we estimate the first expectation as follows, 
    \begin{align}
        \nonumber
        E_1&\le 
        \e_x\left[\e_{X(n_{k+1})}[V_\beta(X(n_{k}-n_{k+1}));\tau>n_k-n_{k+1}];\tau>n_{k+1}, 
        C_{n,k}
        \right]\\
        \label{e1}
        &\le 
        \e_x\left[V_\beta(X(n_{k+1}));\tau>n_{k+1}, 
        C_{n,k}
        \right]=
        E_{n,k+1}(x).  
    \end{align}    
  We split the second expectation as follows,    
  \begin{multline*}
    E_2\le E_{21}+E_{22}  
    :=\e_x\left[V_\beta(X(n_{k+1}));\tau>n_k,
    D_{n,k}\right]\\ 
    +
    \e_x\left[V_\beta(X(n_{k}))-V_\beta(X(n_{k+1}));\tau>n_k,C_{n,k}^c,D_{n,k}\right].
  \end{multline*}  
  Then, using the supermartingale property again, 
  \begin{align*}
    E_{21}&\le 
    \e_x\left[V_\beta(X(n_{k+1}));\tau>n_{k+1},D_{n,k}\right]\\ 
    &\le 
    \e_x\left[V_\beta(X(n_{k+1}));\tau>n_{k+1}\right] 
    \sup_{w\in K} \pr_w\left(|X(n_k-n_{k+1})-w|> \frac{1}{4}\widetilde Ac_k\sqrt{n_k}\right)\\ 
    &\le V_\beta(x) \sup_{w\in K} \pr_w\left(|X(n_k-n_{k+1})-w|> \frac{1}{4}\widetilde Ac_k\sqrt{n_k}\right). 
  \end{align*} 
Now,  using Assumption \ref{assumption:x}, Chebyshev's inequality and Lemma~\ref{lemma:fuk} with 
$z= \frac{1}{4} \widetilde Ac_k\sqrt{n_k}$ and $y= \frac{z}{\sqrt d}$, we obtain 
\begin{multline*}
    \sup_{w\in K} \pr_w\left(|X(n_k-n_{k+1})-w|
    \ge \frac{1}{4}\widetilde Ac_k\sqrt{n_k}\right) \le 
    \frac{16 ed^2}{\widetilde A^2 c_k^{2} } +  
    n_{k+1}\pr\left(Y> \frac{1}{4\sqrt d}\widetilde A c_k\sqrt{n_k}\right) \\ 
    \le  \frac{16 ed^2}{\widetilde A^2 c_k^{2} } + 
    \frac{ 16d \e\left[|Y|^2 \log{(1+|Y|)}; |Y|>c\widetilde A\sqrt{n_k}\right]}{\widetilde A^2 c_k^{2}  \log(n_{k})}
    \le  \frac{C}{\widetilde A^2 c_k^{2}}. 
\end{multline*}
Hence,
\begin{equation}\label{eq21} 
    E_{21}\le 
    \frac{CV_\beta(x)}{\widetilde A^2 c_k^{2}}.  
\end{equation}    
Next note that if follows from Lemma~\ref{diff-bound} and 
Lemma~\ref{lem:u.beta-diff} that 
\begin{multline*}
    V_\beta(X(n_{k}))-V_\beta(X(n_{k+1})) \\   
    \le 
    C\left(|X(n_{k})-X(n_{k+1})|^p 
    +|X(n_{k})-X(n_{k+1})||Rx_0+X(n_{k+1})|^{p-1}\right).  
\end{multline*}
Then, for the term $E_{22}$ we obtain, 
\begin{multline*} 
    E_{22}
    \le 
    C \e_x\left[|X(n_{k})-X(n_{k+1})|^p,\tau>n_{k},D_{n,k}\right]\\ 
    +C \e_x\left[|Rx_0+X(n_{k+1})|^{p-1}|X(n_{k})-X(n_{k+1})|,\tau>n_k, C_{n,k}^c,D_{n,k}\right]\\ 
\le 
    C\pr_x(\tau>n_{k+1})
    \sup_{w\in K}\e_w\left[|X(n_k-n_{k+1})-w|^p,|X(n_k-n_{k+1})-w|> \frac{1}{4}\widetilde Ac_k\sqrt{n_k}\right]\\ 
    +
    C 
    \left(\widetilde A c_{k+1}\sqrt{n_{k+1}}+|x|\right)^{p-1}
    \pr_x(\tau>n_{k+1})\\ 
\times \sup_{w\in K}\e_w\left[|X(n_k-n_{k+1})-w|,|X(n_{k+1})-w|> \frac{1}{4}\widetilde Ac_k\sqrt{n_k}\right]  
\le E_{221}+E_{222},
\end{multline*}    
where, making use of Lemma~\ref{lemma:bound.p.tau}, 
\begin{multline*} 
    E_{221}
    \le 
    CV_\beta(x)
    \sup_{w\in K}
    \Biggl(
	    \frac{\e_w\left[|X(n_k-n_{k+1})-w|^p,|X(n_k-n_{k+1})-w|>\frac14 \widetilde A c_{k}\sqrt{n_k}\right]}{(n_{k+1})^{p/2}}\\ 
	    +(\widetilde A c_{k+1}\sqrt{n_{k+1}} )^{p-1}
\frac{\e_w\left[|X(n_k-n_{k+1})-w|,|X(n_k-n_{k+1})-w|> \frac14 \widetilde A c_k\sqrt{n_k}\right]}{(n_{k+1})^{p/2}}
    \Biggr)
\end{multline*}   
and, then using Corollary~\ref{lemma:bound.p.tau.1},  
\begin{multline*} 
    E_{222}
    \le 
    C|x|^{p-1} 
    d(Rx_0+x)
    \sup_{w\in K}
    \left(
        \frac{\e_w\left[|X(n_k-n_{k+1})-w|,|X(n_k-n_{k+1})-w|>\frac14 \widetilde A c_k\sqrt{n_k}\right]}{\sqrt {n_{k+1}}}
    \right).
\end{multline*}
Since $p\ge 1$ we can continue as follows, 
\begin{equation*}
    E_{221}
    \le 
    CV_\beta(x)
    \frac{\sup_{w\in K} \e_w\left[|X(n_k-n_{k+1})-w|^p,|X(n_k-n_{k+1})-w|> \frac14\widetilde Ac_k\sqrt{n_k}\right]}{(n_{k+1})^{p/2}}.
\end{equation*}  
Observe now an inequality for a non-negative random variable \(Z\), 
\begin{multline}\label{expect.tail}
    \e[Z^p;Z>t]=p\int_t^{\infty}\pr(Z>z)z^{p-1}dz +t^p\pr(Z>t)
    \le p\int_{t/2}^{\infty}\pr(Z>z)z^{p-1}dz \\ 
    +\frac{1}{p(1-2^{-p})}\int_{t/2}^t\pr(Z>z) z^{p-1}dz
    \le C_p\int_{t/2}^{\infty}\pr(Z>z)z^{p-1}dz. 
\end{multline}  
Next fix \(\varepsilon\) such that 
\begin{equation}
    \label{varepsilon}
    p-\frac{2}{\sqrt d\varepsilon}<-2 
\end{equation}    
and put 
\[
f_0(A):=    
\e\left[|Y|^2 \log(1+|Y|), |Y|>\frac18\varepsilon \widetilde A \right]
+
\e[Y^p;Y>\frac18\varepsilon \widetilde A] \downarrow 0,\quad A\to \infty. 
\]
Then, using the  Fuk-Nagaev inequality stated in Lemma~\ref{lemma:fuk} 
with  the truncation level $y = \varepsilon z$,  we obtain,
\begin{multline}
\label{eqn:e22_}
\frac{E_{221}}{V_\beta(x)}
\le \frac{C_p}{(n_{k+1})^{p/2}} 
\sup_{w\in K} \int_{\frac18\widetilde Ac_k\sqrt{n_k}} 
\pr\left(
    |X(n_k-n_{k+1})-w|>z 
\right)z^{p-1}dz\\
\le \frac{C}{(n_{k+1})^{p/2}} \Biggl(\int^\infty_{\frac18 \widetilde Ac_k\sqrt{n_k}} \left(2d e^{\frac{1}{\sqrt d \varepsilon}}\left(\frac{n_{k+1}}{\varepsilon z^2}\right)^{\frac{1}{\sqrt {d} \varepsilon}}\right) z^{p-1}dz\\ +
n_{k+1}\int^\infty_{\frac18 \widetilde Ac_k\sqrt{n_k}} \pr (Y > \varepsilon z) z^{p-1} dz\Biggr). 
\end{multline}
Integrating the first integral, since  $\frac{2}{\sqrt d\varepsilon}-p>2$, 
\begin{multline*}
	\frac{1}{n_{k+1}^{p/2}}\int^\infty_{\frac18 \widetilde Ac_k\sqrt{n_k}} 2d e^{\frac{1}{\sqrt d\varepsilon}}\left(\frac{n_{k+1}}{\varepsilon z^2}\right)^{\frac{1}{\sqrt d\varepsilon}} z^{p-1} dz \\
	\le 
	2d n_{k+1}^{\frac{1}{\sqrt{d}\varepsilon}-\frac{p}{2}}  
	\left(\frac{e}{\varepsilon}\right)^{\frac{1}{\sqrt{d}\varepsilon}}
   \left(\frac{2}{\sqrt d\varepsilon}-p\right)
   \left(\frac18 \widetilde A c_k\sqrt{n_{k+1}}\right)^{p-\frac{2}{\sqrt d\varepsilon}} 
   \le \frac{C}{(c_k\widetilde A)^2}.
\end{multline*}
We estimate the second integral separately for \(p\le 2\) and \(p>2\). 
For $p\le 2$, using  Assumption \ref{assumption:x} and  Markov's inequality,
\begin{multline*}
	\frac{1}{n_{k+1}^{p/2}} \int^\infty_{\frac18\widetilde Ac_k \sqrt{n_k}} n_{k+1} z^{p-1} \pr (|Y|>\varepsilon z) dz\\  
	\le \frac{1}{n_{k+1}^{p/2-1}} \left(\frac{1}{\frac18\widetilde A c_k\sqrt{n_k}}\right)^{2-p} 
	\int^\infty_{\frac18\widetilde A c_k \sqrt{n_k}} z \pr (|Y|>\varepsilon z) dz 
 \\
 \le \left(\frac{C}{\widetilde A c_k}\right)^{2-p}  
 \frac{f_0(A)}{\log(\frac\varepsilon8\widetilde A c_k \sqrt{n_k})} 
\end{multline*}
For $p>2$, we estimate the second integral using  Assumption \ref{assumption:x} and  Markov's inequality,
\[
	\frac{1}{(n_{k+1})^{p/2}} \int^\infty_{\frac18\widetilde A c_k \sqrt{n_k}} n_{k+1} z^{p-1} \pr (|Y|>\varepsilon z) dz \le
\frac{Cf_0(A)}{n_{k}^{p/2-1}} 
 . 
\] 
Therefore, we obtain,
\begin{equation}
    \label{e221}
    E_{221}
    \le 
    C V_\beta(x) \left(\frac{1}{\widetilde A^2 c_k^{2}} + 
    \frac{f_0(A)\ind{p\le 2}}{(\widetilde Ac_k)^{2-p}\log(\frac\varepsilon8\widetilde A c_k \sqrt{n_k})} 
    +\frac{f_0(A)\ind{p>2} }{n_{k}^{p/2-1}}
    \right).
\end{equation}    

Next we will consider \(E_{222}\). 
Similarly, using the Fuk-Nagaev inequality as earlier  we obtain  
\begin{multline*} 
\frac{\sup_{w\in K} \e_y[|X(n_k-n_{k+1})-w|,|X(n_k-n_{k+1})-w|\ge \frac14 \widetilde Ac_k\sqrt{n_k}]}{\sqrt{n_{k+1}}}\\ 
\le 
\frac{C}{\sqrt{n_{k+1}}}
\Biggl(\int^\infty_{\frac18 \widetilde Ac_k\sqrt{n_k}} \left(2d e^{\frac{1}{\sqrt d \varepsilon}}\left(\frac{n_{k+1}}{\varepsilon z^2}\right)^{\frac{1}{\sqrt {d} \varepsilon}}\right)dz\\ +
n_{k+1}\int^\infty_{\frac18 \widetilde Ac_k\sqrt{n_k}} \pr (Y > \varepsilon z)  dz\Biggr)
\end{multline*}
We can estimate the first integral as earlier, by the Markov inequality, 
with   $\varepsilon$ such that $1-\frac{2}{\sqrt d\varepsilon}<-2$, 
\begin{multline*}
	\frac{1}{\sqrt{n_{k+1}}}\int^\infty_{\frac18 \widetilde Ac_k\sqrt{n_k}} 2d e^{\frac{1}{\sqrt d\varepsilon}}\left(\frac{n_{k+1}}{\varepsilon z^2}\right)^{\frac{1}{\sqrt d\varepsilon}} dz \\
	\le 
	2d \left(\frac{e}{\varepsilon}\right)^{\frac{1}{\sqrt{d}\varepsilon}}
    (n_{k+1})^{\frac{1}{\sqrt{d}\varepsilon}-\frac{1}{2}}  
	\left(\frac{2}{\sqrt d\varepsilon}-1\right)
   \left(\frac18 \widetilde A c_k\sqrt{n_{k+1}}\right)^{1-\frac{2}{\sqrt d\varepsilon}}
   \le \frac{C}{(c_k\widetilde A)^{2}}.
\end{multline*} 
For the second integral, as earlier,  
\begin{multline*}
    \frac{1}{\sqrt{n_{k+1}}} \int^\infty_{\frac18\widetilde Ac_k \sqrt{n_k}} n_{k+1}  \pr (|Y|>\varepsilon z) dz\\  
\\
 \le \left(\frac{8}{\widetilde A c_k}\right)  
 \frac{c \e\left[|Y|^2 \log(1+|Y|), |Y|>\frac18\varepsilon \widetilde A c_k \sqrt{n_k}\right]}{\log(\widetilde Ac_k^2 n_{k})} 
 \le \frac{C}{\widetilde A c_k\log(\widetilde Ac_k^2 n_{k})}.
\end{multline*}
Hence, 
\begin{equation} 
    \label{e222}
E_{222}\le V_\beta(x)
\frac{C}{\widetilde Ac_{k} \log(\widetilde Ac_k^2 n_{k})}.
\end{equation}
Combining~\eqref{e221} and \eqref{e222} we obtain 
    \begin{multline}\label{e2}
        E_{n,k}(x) \le E_{n,k+1}(x) 
        +
        V_\beta(x)
        \Biggl(
            \frac{1}{\widetilde A^2 c_k^{2}} + 
            \frac{f_0(A)\ind{p\le 2}}{(\widetilde Ac_k)^{2-p}\log(\frac\varepsilon8\widetilde A c_k \sqrt{n_k})} \\ 
            +\frac{f_0(A)\ind{p>2} }{n_{k}^{p/2-1}}
            +\frac{C}{\widetilde Ac_{k} \log(\widetilde Ac_k^2 n_{k})}
        \Biggr). 
    \end{multline}
We can iterate~\eqref{e1} and~\eqref{e2} now $m$ times to obtain 
   \begin{multline*} 
    \e_x[V_\beta(X(n));\tau>n, |X(n)|> A\sqrt n] \\ 
    \le   
    \e_x\left[V_\beta(X(1));\tau>1, 
    |X(n_m)-x|>\widetilde A c_m \sqrt {n_m}\right]  
    +V_\beta(x)
    \frac{C}{\widetilde A}
    \\
    + 
V_\beta(x)
\sum_{k=0}^{m-1} 
\left(
    \frac{f_0(A)\ind{p\le 2}}{(\widetilde Ac_k)^{2-p}\log(\frac\varepsilon8\widetilde A c_k \sqrt{n_k})} 
            +\frac{f_0(A)\ind{p>2} }{n_{k}^{p/2-1}}
\right)
 . 
\end{multline*} 
Recall that we picked \(m=m(n)\) in such a way 
that \(n_m=1\).   
Then, clearly, 
\[
    \sum_{k=0}^{m-1} 
    \frac{f_0(A)\ind{p>2} }{n_{k}^{p/2-1}} 
    \le 
    f_0(A)\ind{p>2} 
    \sum_{k=0}^{m-1}
    \frac{1}{2^{(k)(p/2-1)}}  
    \le Cf_0(A). 
\]
Next, for \(\widetilde A>\frac{8}{\varepsilon}\), 
\begin{multline*}
    \sum_{k=0}^{m-1} 
    \frac{f_0(A)\ind{p\le 2}}{(\widetilde Ac_k)^{2-p}\log(\frac\varepsilon8\widetilde A c_k \sqrt{n_k})} 
    \le 
    f_0(A) 
    \sum_{k=0}^{m-1} 
    \frac{1}{\log(c_k \sqrt{n_k})} \\ 
    \le C f_0(A) 
    \sum_{k=0}^{m-1} 
    \frac{1}{k\log 2+ (m-k)\log (3\sqrt 2/4)}. 
\end{multline*}
We can split the sum in \(2\) parts and 
obtain, 
\[
    \sum_{k=0}^{m/2} 
    \frac{1}{k\ln 2+ (m-k)\log (3\sqrt 2/4)}
    \le \frac{m/2}{m/2\log (3\sqrt{2}/4)} 
\]
and 
\[
\sum_{k=m/2+1}^{m-1} 
    \frac{1}{k\log 2+ (m-k)\log (3\sqrt 2/4)}
    \le \frac{m/2}{m/2\log (2)}. 
\]
Hence, 
\[
    \sum_{k=0}^{m-1} 
    \frac{f_0(A)\ind{p\le 2}}{(\widetilde Ac_k)^{2-p}\log\left(\frac\varepsilon8\widetilde A c_k \sqrt{n_k}\right)} 
    \le Cf_0(A).   
\]
We are going  to deal with the first term now. 
First 
\begin{multline*} 
    \e_x\left[V_\beta(X(1));\tau>1, 
    |X(1)-x|> \widetilde A c_m\right]\\ 
\le V_\beta(x) \pr_x\left(|X(1)-x|> \widetilde A c_m\right) \\ 
+ 
\e_x\left[V_\beta(X(1))-V_\beta(x);\tau>1, 
    |X(1)-x|>\widetilde A c_m\right]. 
\end{multline*} 
For the first term we obtain immediately, 
\[
    \pr_x\left(|X(1)-x|> \widetilde A c_m\right) 
    \le \frac{\e\left[Y^p;Y>\widetilde A c_m\right]}{(\widetilde A c_m)^p}
    \le f_0(A). 
\]
For the last term 
 \begin{multline*}
    \e_x\left[V_\beta(X(1))-V_\beta(x);\tau>1, 
    |X(1)-x|>\widetilde A c_m\right]\\ 
\le
C |Rx_0+x|^{p-1}
\e_x\left[|X(1)-x|;
    |X(1)-x|>\widetilde A c_m\right]\\
+C\e_x\left[|X(1)-x|^{p};
    |X(1)-x|>\widetilde A c_m\right]
    \le CV_\beta(x)
    f_0(A).
 \end{multline*}  
The statement now follows. 
\end{proof}
\begin{proof}[Proof of Theorem~\ref{thm:p.tau}] 

Part (i) has been proved in Lemma~\ref{lemma:bound.p.tau}.     

We consider now part (iii).
Take $m=m(n)$ such that $m(n)=o(n)$, and for $\varepsilon <1$, $A>1$, then divide the cone into three parts as,
\begin{align*}
    K_1 &:=\{y \in K, d(y) \le \varepsilon \sqrt{m}, |y| \le A \sqrt{m}\} \\K_2 &:=\{y \in K, d(y) > \varepsilon \sqrt{m}, |y| \le A \sqrt{m}\} \\K_3 &:=\{y \in K,|y| > A \sqrt{m}\}.
\end{align*}
Let $D$ be a compact set in the cone.  Using Markov property at $m\colon m <n$, 
    \begin{align*}
        &\mathbf{P}_x \left(\frac{X(n)}{\sqrt{n}} \in D, \tau>n\right) \\&= \int_{K_1} \pr_x(X(m) \in d y, \tau>m) \mathbf{P}_y \left(\frac{X(n-m)}{\sqrt{n}} \in D, \tau > n-m\right) \\&+\int_{K_2} \mathbf{P}_x\left(X(m) \in d y, \tau>m\right) \mathbf{P}_y \left(\frac{X(n-m)}{\sqrt{n}} \in D, \tau > n-m\right)\\&+\int_{K_3} \mathbf{P}_x(X(m) \in d y, \tau>m) \mathbf{P}_y \left(\frac{X(n-m)}{\sqrt{n}} \in D, \tau > n-m\right).
    \end{align*}
Applying  Lemma \ref{lemma:bound.p.tau} we estimate  the integral over $K_1$ as follows,
    \begin{align*}
   & \int_{K_1} \mathbf{P}_x(X(m) \in d y, \tau>m) \mathbf{P}_y \left(\frac{X(n-m)}{\sqrt{n}} \in D, \tau > n-m\right) \\& \le \int_{K_1} \mathbf{P}_x(X(m) \in d y, \tau>m) \mathbf{P}_y (\tau>n-m) \\& \le C \int_{K_1} \mathbf{P}_x(X(m) \in d y, \tau>m) \frac{u(y+Rx_0)}{n^{\frac{p}{2}}}.
    \end{align*}
Applying equation \eqref{eqn:u.main.bound} and making use of the definition of $K_1$,
\[ 
     u(y) \le u(y+Rx_0)  \le u|y+Rx_0|^{p-1} d(y+Rx_0) \le C \varepsilon A^{p-1} m^{\frac{p}{2}}.
\] 
Therefore, we can continue as folllows, 
\begin{equation}
\label{eqn:k1}
    \begin{split}
        &\int_{K_1} \mathbf{P}_x(X(m) \in d y, \tau>m) \mathbf{P}_y \left(\frac{X(n-m)}{\sqrt{n}} \in D, \tau > n-m\right) \\& \le \frac{C}{n^{\frac{p}{2}}} \int_{K_1} \mathbf{P}_x(X(m) \in d y, \tau>m)\varepsilon A^{p-1}m^{\frac{p}{2}} \\& \le \frac{C}{n^{\frac{p}{2}}}\varepsilon A^{p-1}m^{\frac{p}{2}}\mathbf{P}_x(\tau>m) \le \frac{C}{n^{\frac{p}{2}}} \varepsilon A^{p-1} u(x+Rx_0)
    \end{split}
\end{equation}
For $K_3$ we  apply the upper bound for the tail distribution and Proposition~\ref{prop.A}. 
Then
\begin{equation}
    \begin{split}
    & \int_{K_3} \mathbf{P}_x(X(m) \in d y, \tau>m) \mathbf{P}_y \left(\frac{X(n-m)}{\sqrt{n}} \in D, \tau > n-m\right)  \\& \le \frac{C}{n^{\frac{p}{2}}} \int_{K_3} \mathbf{P}_x(X(m) \in dy, \tau>m ) V_\beta(y)\\& \le \frac{C}{n^{\frac{p}{2}}} \e_x[V_\beta(X(m)), \tau>m, |X(m)| > A \sqrt{m}] \\& \le \frac{C}{n^{\frac{p}{2}}} f(A) V_\beta(x) \le \frac{c}{n^{\frac{p}{2}}} f(A) u(x+Rx_0),
    \end{split}
    \label{eqn:bound.k3}
\end{equation}
where $f(A)\downarrow 0$ as $A\uparrow\infty$.  
As a result,  
\begin{equation}
\label{eqn:k3}
    \int_{K_3}\mathbf{P}_x(X(m)\in d y , \tau>m)u(y) \le \frac{C}{n^{\frac{p}{2}}} f(A) u(x+Rx_0). 
\end{equation}
 By taking $m(n)$ sufficiently large (but still $o(n)$), 
we can ensure using 
the functional central limit theorem for martingales 
that uniformly in $y\in K_2$,  
\begin{equation}
    \begin{split}
    \mathbf{P}_y\left(\frac{X(n-m)}{\sqrt{n}} \in D, \tau>n-m\right) & 
    \sim \mathbf{P}\left(B\left(\frac{n-m}{n}\right)\in D, \tau^{bm}_y >n-m\right) \\& \sim C\frac{u(y)}{n^{\frac{p}{2}}} \int_D \exp\left({-\frac{|z|^2}{2}}\right) u(z) d z. 
    \end{split}
\end{equation}
Then, applying  Lemma~\ref{lemma:bound.p.tau}, 
\begin{align*}
   & \int_{K_2} \mathbf{P}_x(X(m) \in d y, \tau>m) 
   \mathbf{P}_y\left(\frac{X(n-m)}{\sqrt{n}} \in D, \tau > n-m\right) \\& \sim \frac{C}{n^{\frac{p}{2}}} \int_{D} e^{-|z|^2/2} u(z) dz\int_{K_2} \mathbf{P}_x(X(m) \in d y, \tau>m) u(y).  
\end{align*}
Using the bounds  of the integrals over  $K_1$ and $K_3$,  
\begin{equation}
\label{eqn:k2}
    \begin{split}
    &\Biggl|\int_{K_2} \mathbf{P}_x\left(X(m) \in d y, \tau>m\right) \mathbf{P}_y\left(\frac{X(m)}{\sqrt{n}} \in D, \tau > n-m\right) \\& - \frac{C}{n^{\frac{p}{2}}} \int_D u(z) \exp\left({-\frac{|z|^2}{2}}\right) d z \int_D \pr_x(X(m)\in dy, \tau>m)u(y) \Biggr| \\& =\left|\int_{K_1 \bigcup K_3}\mathbf{P}_x(X(m) \in d y, \tau>m) u(y)\ \int_D u(z) \exp\left({-\frac{|z|^2}{2}}\right) d z\right| \\& \le \frac{C}{n^{\frac{p}{2}}} \left(\varepsilon A^{p-1}+ f(A)\right) u(x+Rx_0). 
    \end{split}
\end{equation}
Combining equations \eqref{eqn:k1}, \eqref{eqn:k2} and \eqref{eqn:k3}, the statement of part (iii) follows,
\begin{align*}
&\Bigg|\pr_x\left(\frac{X(n)}{\sqrt{n}}\in D,\tau_x>n\right)
-\frac{\varkappa}{n^{p/2}}\int_D u(z)e^{-|z|^2/2}dz
\e[u(X(m));\tau_x>m]\Bigg|\\
&\hspace{1cm} \le Cu(x+Rx_0)\frac{\varepsilon A^{p-1}+f(A)}{n^{p/2}}.
\end{align*}
uniformly in $|x|< \sqrt{m}$.
Letting  $\varepsilon \to 0$ and $A \to \infty$,
\begin{equation}
    \mathbf{P}_x \left( \frac{X(n)}{\sqrt{n}} \in D |\tau >n \right) \sim \frac{\varkappa V(x)}{n^\frac{p}{2}} c\int_D u(z)\exp\left\{-\frac{|z|^2}{2}\right\} d z
\end{equation}
and 
\begin{equation}
    \pr_x(\tau>n) \sim \frac{\varkappa V(x)}{n^\frac{p}{2}}.
\end{equation}

\end{proof}

\section{A version of Fuk-Nagaev inequality}
We now present an adaptation of the Fuk-Nagaev inequality~\cite{fuk1973} 
\begin{lemma}
\label{lemma:fuk}
Let $(X(n))_{n\ge 1}$ be a $d$-dimensional Markov chain, for which
Assumption~\ref{assumption:x} holds and $\e_x[X_j(1)-x]=0$ for $j=1,\ldots,d$ and all $x$. 
Then, for all $z\text{,}y >0$, 
\begin{equation}
\label{fuk1}
    \pr_x \left(|X(n)-x| >z, \max_{k \le n} |X(k)-X(k-1)|<y\right) \le 2 d e^{z/{\sqrt{d}y}}\left(\frac{n\sqrt{d}}{z y}\right)^{z/{\sqrt{d}y}}
\end{equation}
and 
\begin{equation}
\label{fuk2}
    \pr_x (|X(n)-x|>z) \le 2 d e^{z/{\sqrt{d}y}}\left(\frac{n\sqrt{d}}{z y}\right)^{z/{\sqrt{d}y}}+n\pr(Y>y).
\end{equation}
\end{lemma}
\begin{proof}
By bounding with the sum of probabilities for one dimensional projections, we have,
\begin{equation*}
\begin{split}
    &\pr_x \left(|X(n)-x| >z, \max_{k \le n} |X(k)-X(k-1)|<y\right) \\& \le \sum^d_{j=1} \pr_x \left(|X_j(n)-x_j| > \frac{z}{\sqrt{d}}, \max_{k \le n} |X_j(k)-X_j(k-1)|<y\right) 
    \end{split}
\end{equation*}
Assumption \ref{assumption:x} implies that the equation (20) in \cite{fuk1973} holds,
\begin{equation*}
    \e[|X_j(n)-X_j(n-1)|^2\mid \mathcal{F}_{n-1}]
    =\e_{X(n-1)}[X_j(1)-X_j(0)]
    \le \e [|Y|^2] =: b_i^2.
\end{equation*}
Then we can apply Corollary 1 from \cite{fuk1973} 
with $B_n^2:=\sum_{i=1}^n b_i^2$,
\begin{equation}
\label{eqn:fuk2}
\begin{split}
    &\pr_x \left(|X_j(n)-x_j| >\frac{z}{\sqrt{d}}\right) \\&\le \sum^n_{i=1}\pr_x\left (|X_j(i)-X_j(i-1)| \ge y\right)+2 \exp \left\{\frac{z}{\sqrt{d} y}-\left(\frac{z}{\sqrt{d} y}+\frac{B_n^2}{y^2}\right) \log{\left(\frac{z y}{\sqrt{d} B^2_n}+1\right)}\right\}.
    \end{split}
\end{equation}
After rearrangement of terms,
\begin{equation}
  \begin{split}
 &\pr_x \left(|X(n)-x| > \frac{z}{\sqrt{d}}, \max_{k \le n} |X(k)-X(k-1)|<y\right) \\& \le 2 d \exp \left\{\frac{z}{\sqrt{d} y}-\left(\frac{z}{\sqrt{d} y}+\frac{B_n^2}{y^2}\right) \log{\left(\frac{z y}{\sqrt{d} B^2_n}+1\right)}\right\} \\& \le 2d e^{\frac{z}{\sqrt{d} y}} \left(\frac{\sqrt{d} n \e [|Y|^2}{ z y}\right)^{\frac{z}{\sqrt{d} y}}
  \end{split}
\end{equation}
Then we get the first inequality, and the equation \eqref{eqn:fuk2} also gives the second inequality.
\end{proof}
Consequently, Theorem \ref{thm:p.tau} can be proved in the similar way of random walk case \cite{randomwalkinconesrevisited} with Lemma \ref{lemma:fuk}.


    \end{document}